\RequirePackage{fix-cm}
\documentclass[smallextended]{svjour3}       
\smartqed  
\usepackage{amssymb}
\usepackage{amsmath, amscd, amsfonts}
\usepackage{graphicx}
\usepackage{latexsym}
\usepackage{float}
\usepackage[numbers,compress]{natbib}
\usepackage{caption}
\usepackage{float}

\numberwithin{equation}{section}

\renewcommand{\Im}{\operatorname{Im}}

\usepackage{color}
\definecolor{blue1}{rgb}{0.0, 0.0, 1.0}
\definecolor{gray}{rgb}{0.9,0.9,0.9}
\definecolor{gray1}{rgb}{0.7,0.7,0.7}
\definecolor{gray2}{rgb}{0.8,0.8,0.8}
\definecolor{magenta}{rgb}{1.0, 0.0, 1.0}
\definecolor{purple}{rgb}{1.0, 0.2, .6}

\usepackage{hyperref}
\hypersetup{ colorlinks=true, urlcolor  = blue, linkcolor = blue,
citecolor = blue1,}
 \journalname{Theory in Biosciences}

\begin{document}

\title{ Dynamics of neural fields with exponential temporal kernel
}
\subtitle{}

\author{Elham Shamsara \and Marius E. Yamakou \and Fatihcan M. Atay \and J\"{u}rgen Jost
}


\institute{E. Shamsara \at Methods in Medical Informatics, Department of Computer Science, University of Tübingen, 72076 Tübingen, Germany\\
M. E. Yamakou \at Department of Data Science, Friedrich-Alexander-Universit\"{a}t Erlangen-N\"{u}rnberg, Cauerstr. 11, 91058 Erlangen, Germany\\
Fatihcan M. Atay \at Department of Mathematics, Bilkent University, 06800 Ankara, Turkey\\
J. Jost \at Max-Planck-Institut
f\"{u}r Mathematik in den Naturwissenschaften, Inselstr. 22, 04103 Leipzig, Germany
\at Santa Fe Institute for the Sciences of Complexity,  Santa Fe, NM 87501, USA
\at ScaDS.AI, Dresden/Leipzig, Germany\\
             \email{elham.shamsara@uni-tuebingen.de}\\
             \email{marius.yamakou@fau.de}, Corresponding author\\
             \email{f.atay@bilkent.edu.tr}\\
             \email{jost@mis.mpg.de}
}

\date{Received: date / Accepted: date}

\maketitle

\begin{abstract}
We consider the standard neural field equation with an exponential temporal kernel. 
 We analyze the time-independent (static) and time-dependent (dynamic) bifurcations of the equilibrium 
 solution and the emerging spatiotemporal wave patterns. We show that an exponential temporal kernel 
 does not allow static bifurcations such as saddle-node, pitchfork, and in particular, static Turing 
 bifurcations. However, the exponential temporal kernel possesses the important property that it takes into account the finite memory of past activities of neurons, which Green's function does not. Through a dynamic bifurcation analysis,  we give explicit bifurcation conditions. Hopf bifurcations  lead to temporally non-constant,  but spatially constant solutions, 
 but Turing-Hopf bifurcations generate spatially and temporally non-constant solutions, in particular, traveling waves. Bifurcation parameters are 
 the coefficient of the exponential temporal kernel, 
 the transmission speed of neural signals, the time delay rate of synapses, and the ratio of 
 excitatory to inhibitory synaptic weights. 

\keywords{Neural fields, exponential temporal kernel, leakage, transmission delays, bifurcation analysis, spatiotemporal patterns}
\end{abstract}

\newpage


\section{Introduction}\label{section1}

It is a well-established and basic neurophysiological fact that neural
activity leads to particular spatiotemporal patterns in the cortex; see for
instance the survey in \citep{wu2008},  and other
brain structures like the hippocampus, see for example
\citep{lubenov2009}. These spatiotemporal patterns have the qualitative
properties of periodic or traveling waves, see, for instance, \citep{townsend2015}. 
Such patterns, like  periodic and
traveling waves, play important roles in neurophysiological models  of  cognitive processing, beginning with 
the synchronization models of von der Malsburg \citep{malsburg1981} or the synfire chains of Abeles \citep{abeles1982}. 
It is, therefore, important to understand 
the emergence of these patterns in densely connected 
 networks of neurons that communicate 
with each other by transmitting neural information via their synapses
\citep{kandel2000principles}. For  understanding such
macroscopic patterns, it seems natural to abstract from details at the
microscopic, that is, neuronal, level, and to study pattern formation from a
more general perspective. One recent approach \citep{galinsky2020} looks at the
electromagnetic properties and the folding geometry of brain tissue. A more
classical, and by now rather well-established approach is the neural field
theory.  
Neural field theory considers populations of neurons embedded in a coarse-grained spatial area, and neural field 
equations describe the spatiotemporal evolution of coarse-grained variables like the  firing rate activity in these populations of neurons \citep{wilson1973mathematical}. 
Wilson and Cowan first introduced neural field models as a spatially extended version of 
Hopfield neural networks \citep{wilson1973mathematical,wilson1972excitatory}. 
A simplified model that could be mathematically treated in a rather explicit form was developed by Amari \citep{amari1977dynamics}, 
which consists of nonlinear integrodifferential equations.
These equations 
play an important role also in other fields, such as machine learning, which combines ideas from neural field modeling 
and model based recognition \citep{veltz2011stability,perlovsky2006toward}. 

Neural fields have seen significant progress in both theoretical and numerical studies over the recent years 
\citep{alswaihli2018kernel, abbassian2012neural,bressloff2011spatiotemporal,haken2007brain,karbowski2000multispikes,morelli2004associative,prager2003stochastic,spiridon2001effect}.
An important fact in neural field modeling is the consideration of axonal conduction delays arising from the finite speed of signals traveling along the axonal distance.
Some recent and significant contributions to neural field modeling with transmission delays are presented 
in \citep{atay2004stability, atay2006neural,hutt2005analysis,hutt2006effects,veltz2011stability,spek2022dynamics,van2013local}. 
In \citep{atay2004stability}, a stability analysis is given for neural field equations in the
presence of finite propagation speed and for a general class of connectivity kernels, and
sufficient conditions for the stability of equilibrium solutions are given.
It is shown that the non-stationary bifurcations of equilibria depend on the propagation delays and the connectivity kernel, 
whereas the stationary bifurcations rely only on the connectivity kernel.
In \citep{hutt2005analysis}, the stability of neural fields with a general connectivity kernel and space-dependent transmission delays is analyzed. 
It is found that  Turing instability  occurs with local inhibition and lateral excitation, 
while wave instability occurs with local excitation and lateral inhibition.
The standard neural field model with propagation speed distribution of signal transmission 
speeds is considered in \citep{hutt2006effects}, where
the effect of distributed speeds on the dynamical behavior is investigated. 
It is shown that the variance of the speed distribution affects the frequency of bifurcating periodic solutions and the phase speed of traveling waves.
It is also shown that the axonal speed distributions increase the traveling front speed. The results in \citep{hutt2006effects} were
extended in \citep{atay2006neural}, where long-range feedback delays are considered in the standard neural field model. 
There, it is shown that in a reduced model, delayed excitatory feedback generally facilitates stationary bifurcations and
Turing patterns while suppressing the bifurcation of periodic solutions and traveling waves. In the case of oscillatory bifurcations, the variance of the distributed propagation and feedback delays affect the frequency of periodic solutions and their traveling speed \citep{muller2018cortical,watt2009traveling}. 

The study in \citep{veltz2011stability}  considers neural field equations with space-dependent delays and uses two techniques: (i) the computation of the eigenvalues of the linear operator defined by the linearized equations to establish a sufficient condition for stability, which is independent of the characteristics of the delays, and (ii) the formulation of the problem as a fixed point problem to find new sufficient conditions for the stability of stationary solutions which depend upon the values of the delays. The work in \citep{spek2022dynamics} investigated a neural field model that incorporates transmission delays and a connectivity kernel consisting of a linear combination of exponentials. The authors examined the dynamics and stability of this model within a two-dimensional spatial domain. They analyzed the spectrum of the linearized equation and identified the presence of a supercritical Hopf bifurcation. They also explored the possibilities of extending this model to incorporate multiple populations and higher-dimensional spatial domains. Moreover, the investigation in \citep{van2013local} showed that one could recast neural field models with transmission delays into abstract delay differential equations (DDE) and subsequently use standard results from dynamical systems theory, such as the principle of linearised (in)stability, center manifold reduction and normal form computation to study the bifurcations of the delayed neural field models. In particular, they showed that the associated steady state of the DDE might destabilize under certain conditions via a Hopf bifurcation.

Various experimental methods of recording the activity of brain tissue in vitro and in vivo 
 demonstrate the existence of traveling waves. Neural field theory offers a theoretical framework 
 for studying such phenomena. The question, then, is to identify the structural 
 assumptions and the parameter regimes for the emergence of traveling waves in neural fields. 	
This work aims to analytically and numerically study the static and dynamic bifurcations and 
spatiotemporal wave patterns generated by the classical neural field model with an \textit{exponential} temporal kernel which is more general than the Green's function used in \citep{atay2004stability} and \citep{senk2020conditions}.
In \citep{senk2020conditions}, the temporal connectivity kernel is the product of an alpha function\footnote{See Section \ref{section2} for a definition of the alpha function.} and the Heaviside function, which yields 
a function with the same properties as the Green's function, and thus yields the same characteristic polynomial as in \citep{atay2004stability}. 
We recall that Green's function $G(t,t')$ is the solution to $LG(t,t')=\delta(t-t')$ satisfying the given boundary conditions, where $L$ is a
differential operator. This is a differential equation for $G$ (or a partial differential equation if we are in more than one dimension), 
with a very specific source term on the right-hand-side: the Dirac delta, which is
$0$ if $t\neq t'$, 
and hence does not consider finite memory of past activities of neurons.

In contrast, in this paper, 
the derivative of the exponential temporal kernel tends to $0$ as $t\rightarrow\infty$. Also, it decreases monotonically in finite time,
meaning that it takes into account a finite memory of past activities of neurons, which Green's function does not. 
Ref.~\citep{senk2020conditions} is quite inspiring for reducing a biologically more realistic microscopic model of leaky integrate-and-fire
 neurons with distance-dependent connectivity to an effective neural field model. Because of the type of kernels used there, 
two different neuron populations, excitatory and inhibitory ones, are needed to induce dynamic bifurcations. 
 Here, we work with a Mexican hat-type spatial kernel (which models short-range excitation and mid-range inhibition), and an exponential temporal kernel, and we will demonstrate similar types of dynamic bifurcations as in 
 \citep{amari1977dynamics,senk2020conditions} with only a single population. Thus, in our model, we have identified the parameter regimes for periodic patterns via Hopf bifurcations and traveling-wave-like spatiotemporal patterns via Turing-Hopf bifurcations. These patterns are typically seen in electrophysiological recordings of the activity of cortical and other brain tissues and may support basic cognitive processes at the neurophysiological level.

This paper is organized as follows: In  Sect. \ref{section2}, we present the model equation and obtain its equilibrium solution. Sect. \ref{section3} is devoted to the static bifurcation analysis of the equilibrium solution. 
In Sect. \ref{section4}, we investigate dynamic bifurcations of the equilibrium solution and the ensuing patterns of traveling waves, and finally, we conclude with a general discussion and some remarks in Sect. \ref{section5}.


\section{The model and the equilibrium solution}\label{section2}
We consider  a neural field model represented by an
infinite-dimensional dynamical system in the form of an
integrodifferential equation \citep{polner2017space, arqub2017adaptation,faugeras2015stochastic, rankin2014continuation}, with axonal conduction
delay  \citep{fang2016monotone, breakspear2017dynamic, pinto2001spatially}. In this equation, the position of a neuron at a time
$t$ is given by a spatial variable $x$, in the literature usually considered to be
continuous in $\mathbb{R}$ or $\mathbb{R}^2$. The state of the
neural field, $v(x,t)$ (membrane potential), evolves according to
\begin{align} \label{eq1}
v(x,t)=\int_{-\infty}^{t} \kappa(t-s)S(x,s)\,ds+\int_{-\infty}^{t} \big[I_1(x,s)-\frac{1}{\tau}v(x,s)\big]\,ds,
\end{align} 
with an arbitrary initial condition $v(x, -\infty) = v_i$. 
Here, $v(x,t)$ is interpreted as a neural field representing the local activity of a population 
of neurons at position $x$ and time $t$, and $I_1(x,t)$ is an external input current originating from the surrounding environment or other neural populations.  e.g., from  other cortical regions or the midbrain \citep{nunez1995experimental}.
The first integral converts the incoming pulse activity $S$ of the neuron at $x$ into its state by
convolution with a temporal kernel (impulse response function)
$\kappa$. The second integral balances the external input $I_1$ with a decay or
leakage term, with a time constant $\tau>0$ arising from the
the temporal decay rate of synapses. 
In this paper, we take the past activity of neurons into account for the
impulse response using an \textit{exponential} temporal kernel. In \citep{atay2004stability}, 
such a kernel was taken as the Green's function of a first-order differential operator. 
Here, in order to be able to carry out a detailed bifurcation analysis depending on that kernel, we use a more explicit form, 
namely an exponential decay:
\begin{align} \label{eq00}
\kappa(t-s)=\begin{cases}
\alpha_{1}e^{-\alpha_{2}(t-s)} & \text{if}\:\:t-s\geq0,\\
0 & \text{if}\:\:t-s<0,
\end{cases}
\end{align}
where $\alpha_{1}$ and
$\alpha_{2}$ are positive constants. A normalization condition requiring that the integral of the kernel be $1$ gives $\alpha_{1}=\alpha_{2}:=\alpha$.
 Such kernels are standard in the neuroscience literature and are usually 
called $\alpha$-functions (see e.g. \citep{gerstner2002spiking}).
We want to explicitly point out that even though the exponential kernel used in this work reduces to the kernel used in \citep{atay2004stability} 
as $\alpha\rightarrow \infty$, our bifurcations cannot, in general, automatically reduce to 
those in \citep{atay2004stability}. The presence of a leakage term (which is neglected in \citep{atay2004stability}), 
characterized by a temporal decay rate parameter $\tau>0$ in our model, does not allow for such a reduction.

The crucial idea in neural field models is that the incoming activity $S(x,t)$ is obtained by a
spatial convolution via an integral with some convolution kernel
$J(x,y)$, that is,
\begin{align} \label{eq0}
S(x,t)=c\int_{\Omega}
J(x,y)F\Big(v\big(y,t-\frac{|y-x|}{\nu}\big)\Big)dy+I_2(x,t).
\end{align}
Here, $c>0$ is some constant that involves various temporal and
spatial scales, $\Omega$ is the spatial domain
which is usually taken as
$\mathbb{R}$ or $\mathbb{R}^2$ in the literature, although other choices, 
like $\mathbb{R}^3$ or $S^2$, are neuro-biologically plausible and mathematically tractable. 
The synaptic weight function $J$ typically describes local excitation--lateral inhibition or
local inhibition--lateral excitation. The function $F$ is a
the transfer function (for instance, a sigmoid or a Heaviside
function $H(v-v_\mathrm{th})$, for some threshold $v_\mathrm{th}$; however, later on, $F$ needs to be smooth), $I_2(x,t)$ is an internal input current originating from the synaptic connectivity between the neurons in the neural population. When a neuron fires, it releases neurotransmitters into the synaptic cleft, influencing the neighboring neurons by exciting or inhibiting their activity. This synaptic interaction results in an internal input current  $I_2(x,t)$ (i.e., feedback from within the neural population itself) that affects the rate of change of neural activity at each point in space. It is worth noting that, unlike most studies where only the external input current is considered, the presence of external and internal input currents (i.e., $I_1(x,t)$ and $I_2(x,t)$, respectively) in neural field modeling allows for a more comprehensive description of how neural activity evolves over space and time, considering both the impact of external stimuli and the intrinsic interactions among neurons within the population.
The distance $|y-x|$ between $x$ and $y$ (for instance the Euclidean distance), and $\nu>0$ is the transmission speed of neural signals.
Thus, a finite transmission speed introduces a distance-dependent
transmission delay, which approaches $0$
as $\nu\rightarrow\infty$. We also assume a homogeneous field where the connectivity $J(x,y)$ 
depends only on the distance $|y-x|$, and so we replace $J(x,y)$ by an even function $J(y-x)$.
In our numerical investigations, we will use the following spatial {convolution kernel \citep{hutt2005analysis} and 
sigmoid transfer function \citep{wilson1973mathematical,robinson1997propagation}:} 
\begin{eqnarray}\label{num}
J(y-x)&=&\frac{a_{e}}{2}e^{-|y-x|}-\frac{a_{i}}{2}re^{-|y-x|r},\\
\label{num1}F(v)&=&\frac{1}{1+\exp(-1.8(v-3))},
\end{eqnarray}
{respectively, where $a_e$ and $a_i$ respectively denote the excitatory and inhibitory synaptic weights and  
$r$ denotes the relation of excitatory and inhibitory spatial ranges \citep{hutt2003pattern}. 
The combination of excitatory and inhibitory axonal networks may yield four different spatial interactions, 
namely pure excitation (i.e., when $a_i=0$), pure inhibition (i.e., when $a_e=0$), local excitation-lateral inhibition 
(i.e., when $a_e\neq0$, $a_i\neq0$, and $r<1$) giving $J$ a Mexican-hat shape, 
and local inhibition-lateral excitation (i.e., when $a_e\neq0$, $a_i\neq0$, and $r>1$, giving $J$ an inverse Mexican-hat shape.}

Differentiating \eqref{eq1} with respect to $t$ yields
\begin{align} \label{eq2}
\frac{\mathrm{d} }{\mathrm{d} t}v(x,t)=\int_{-\infty}^{t} \frac{\mathrm{d}\kappa(t-s) }
{\mathrm{d} t}S(x,s)\,ds-\frac{1}{\tau} v(x,t)+\alpha S(x,t)
+ I_1(x,t).
\end{align}
Inserting \eqref{eq00} and \eqref{eq0} in \eqref{eq2} gives \begin{align} \label{eq6}
\frac{\mathrm{d} }{\mathrm{d}
t}v(x,t)=&-\alpha^{2}c\int_{-\infty}^{t}
e^{-\alpha(t-s)}\int_{\Omega}J(y-x)F\Big(v\big(y,s-\frac{|y-x|}{\nu}\big)\Big)
dy \, ds-\alpha^{2}\int_{-\infty}^{t} e^{-\alpha(t-s)}I_2(x,s)ds\nonumber\\
&-\frac{1}{\tau} v(x,t)+\alpha c \int_{\Omega}
J(y-x)F\Big(v\big(y,t-\frac{|y-x|}{\nu}\big)\Big)dy+\alpha
I_2(x,t) + I_1(x,t).
\end{align}
In order to analyze the dynamic behavior of \eqref{eq6}, we
assume  constant internal and external input currents, i.e., $I_1(x,s)=E$, $I_2(x,s)=I_{0}$, and a
constant solution
\begin{align} \label{eq7}
v(x,t)=v_{0}.
\end{align}
Substituting into \eqref{eq6} shows that $v_0$ satisfies the fixed point equation
\begin{align} \label{eq8}
v_{0}+\alpha\tau cF(v_{0})\int_{\Omega} J(y-x)dy+\alpha\tau I_{0}-\alpha \tau I_{0}-\alpha \tau c F(v_{0})\int_{\Omega} J(y-x)dy
-\tau E=0,
\end{align}
which is satisfied by the fixed point (equilibrium solution)
\begin{align} \label{eq9}
v_{0}=\tau E.
\end{align}
The other terms in \eqref{eq8} cancel because if $v$ and $I_2$ are
  constant, then so is $S$, and hence the first integral in \eqref{eq1} is
  independent of $t$. In the following sections, we study the static and dynamic bifurcations of this equilibrium solution.


\section{Static bifurcations of the equilibrium solution} \label{section3}
For the purpose of this paper, we will consider only one spatial dimension, i.e., we take $\Omega=\mathbb{R}$. 
However, one should note that the results presented in this paper may not automatically translate to higher 
dimensions  (i.e., $\Omega=\mathbb{R}^N, N\geq2$) where entirely new dynamical behaviors could emerge.

To obtain the parametric region of the stability of the equilibrium solution \eqref{eq9}, we
linearize the integrodifferential equation \eqref{eq6} around the equilibrium 
solution $v_{0}=\tau E$. Let $w(x,t)=v(x,t)-v_{0}$. Then,
\begin{align} \label{eq12}
\frac{\mathrm{d} }{\mathrm{d}
t}w(x,t)=&-\alpha^{2}c\int_{-\infty}^{t}
e^{-\alpha(t-s)}\int_{-\infty}^{\infty}J(y-x)
\Big[F(v_{0})+F'(v_0)w\big(y,s-\frac{|y-x|}{\nu}\big)\Big]\,dy\,ds\nonumber\\
&-\alpha I_{0}-\frac{1}{\tau}w(x,t)-\frac{v_0}{\tau} +\alpha
c\int_{-\infty}^{\infty}J(y-x)\Big[F(v_{0})+F'(v_0)w\big(y,t-\frac{|y-x|}{\nu}\big)\Big]dy+\alpha I_{0}+E,
\end{align}
which simplifies to
\begin{align} \label{eq120}
\frac{\mathrm{d} }{\mathrm{d} t}w(x,t)=&-\alpha^{2}c
F'(v_0)\int_{-\infty}^{t}
e^{-\alpha(t-s)}
\int_{-\infty}^{\infty}J(y-x)w\big(y,s-\frac{|y-x|}{\nu}\big)\,dy\,ds\nonumber\\
&-\frac{1}{\tau}w(x,t)+\alpha c F'(v_0)\int_{-\infty}^{\infty}J(y-x)w\big(y,t-\frac{|y-x|}{\nu}\big)\,dy+E - \frac{v_0}{\tau}.
\end{align}

To check the stability of the equilibrium solution, we substitute the  general Fourier-Laplace ansatz for linear integrodifferential equations
\begin{align} \label{ansatz}
 w(x,t)=e^{\lambda t}e^{ikx}, \:\: \lambda \in \mathbb{C}, k \in \mathbb{R},
 \end{align}
 into  \eqref{eq120} to get
\begin{align} \label{eq13}
(\tau\lambda+1)e^{\lambda t}e^{ikx}=&-\alpha^{2}c\tau
F'(v_0)
\int_{-\infty}^{t}e^{-\alpha(t-s)}e^{\lambda s}\,ds\int_{-\infty}^{\infty} J(y-x) e^{iky}e^{-\lambda\frac{|y-x|}{\nu}}\,dy\nonumber\\
&+\alpha c \tau F'(v_0)e^{\lambda
t}\int_{-\infty}^{\infty} J(y-x)
e^{iky}e^{-\lambda\frac{|y-x|}{\nu}}dy+\tau E -v_0 .
\end{align}
From \eqref{eq9}, we note that $\tau E-v_0=0$ in \eqref{eq13}. The first integral in \eqref{eq13} is independent of $y$, so
\begin{align} \label{eq130}
\int_{-\infty}^{t} e^{-\alpha(t-s)}e^{\lambda s}ds=\frac{1}{\alpha+\lambda}e^{\lambda t}, \:\:\: \alpha + \lambda \neq 0.
\end{align}

By making a change of variable: $z=y-x$ in \eqref{eq13}, and
considering $\lambda \neq - \alpha$, we obtain the linear
variational equation as
\begin{align} \label{eq14a}
\tau\lambda+1=\alpha c \tau F'(v_0)
\Big(\frac{\lambda}{\alpha+\lambda}\Big)\int_{-\infty}^{\infty}J(z)e^{-\lambda\frac{|z|}{\nu}}e^{-ikz}dz.
\end{align}
For static bifurcations, that is, for bifurcations leading to 
temporally constant solutions, we must have $\lambda = 0$
\citep{atay2004stability}. However,  $\lambda=0$ is not a
solution of \eqref{eq14a}. And in fact, for a constant solution $v_0\neq \tau E$, the second integral in \eqref{eq1} would diverge. 
Therefore, static bifurcations (such as saddle-node and pitchfork bifurcations) cannot occur, and 
hence, in particular, we conclude the following:

\begin{theorem}
  The neural field equation \eqref{eq1} with the exponential temporal kernel \eqref{eq00} does not admit static  
  bifurcations from the spatially uniform equilibrium solution \eqref{eq7}.  
\end{theorem}

 This is in contrast to the models investigated in  \citep{coombes2007modeling, hutt2003pattern} where such patterns may occur.

\section{A stability condition for constant equilibrium}
We next give a sufficient condition for the
asymptotic stability of the equilibrium solution $v_0$. 
We make use of the following lemma from \citep{atay2004stability}.

\begin{lemma} \label{lm1}
Let $L(\lambda)$ be a polynomial whose roots have non-positive real parts.
Then $ \left | L(\sigma +i\omega ) \right |\geqslant \left |
L(i\omega ) \right | $ for all $\sigma \geqslant 0$ and $\omega
\in \mathbb{R}$.
\end{lemma}

\begin{theorem} \label{theo1}
Let $ D := |\beta|\int_{-\infty}^{\infty}\left | J(z) \right |dz $, where $\beta=\alpha c \tau F'(v_0)$, and $L(\lambda)$ be a polynomial whose roots have non-positive real parts. If 
\begin{align} \label{the1}
D< \underset{\omega \in \mathbb{R}}{\min} \left |L(i\omega) \right|,
\end{align}
then $v_{0}$ is locally asymptotically stable. In particular, the condition
\begin{align} \label{the2}
D<1,
\end{align}
 is sufficient for the local asymptotic stability of $v_{0}$. 
\end{theorem}
\begin{proof}: In the ansatz $w(x,t)=e^{\lambda t}e^{ikx}$, let $\lambda=\sigma+i\omega$, where $\sigma$ and $\omega$ are real numbers. 
We will prove that $\sigma<0$ if \eqref{the1} holds. Suppose, by contradiction,  that \eqref{the1} holds but $\sigma\geq0$. 
Let $L(\lambda):=\tau\lambda + 1$, then by \eqref{eq14a},
it follows that
\begin{eqnarray}\nonumber
\begin{split}
 \left | L(\sigma +i\omega ) \right |&=\left | \beta \right |\left
|\frac{\sigma+i\omega}{\alpha+ \sigma+i\omega}\right |\left |
\int_{-\infty}^{\infty}J(z)e^{-(\sigma
+i\omega)\frac{|z|}{\nu}}e^{-ikz} dz\right| \\
&\leqslant \left| \beta \right| \left| \frac{\sigma+i\omega}{\alpha+ \sigma+i\omega} \right|
\int_{-\infty}^{\infty} \left| J(z)\right |\left |e^{-(\sigma
+i\omega)\frac{|z|}{\nu}}\right|dz \\
&\leqslant \left | \beta \right |\left
|\frac{\sigma+i\omega}{\alpha+ \sigma+i\omega}\right|
\int_{-\infty}^{\infty}\left |J(z)\right|dz.
\end{split}
\end{eqnarray}
Since $\alpha>0$, we have  $\left|\frac{\sigma+i\omega}{\alpha+ \sigma+i\omega}\right|<1$, which means
\begin{eqnarray}
\left | L(\sigma +i\omega ) \right |< \left| \beta \right| \int_{-\infty}^{\infty}\left |J(z)\right|dz.
\end{eqnarray}
By Lemma \ref{lm1},
\begin{align} \label{eq:001}
\left | L(i\omega ) \right |\leq \left | L(\sigma+i\omega ) \right |< \left | \beta \right |
\int_{-\infty}^{\infty}\left |J(z)\right|dz=D,
\end{align}
for some $\omega \in \mathbb{R}$. This, however, contradicts \eqref{the1}. 

Thus $\sigma<0$, and the equilibrium solution $v_0$ is locally asymptotically stable. 
This proves the first statement of the theorem. 
From $L(\lambda):=\tau\lambda + 1$, 
one has $\left|L(i\omega)\right|^2=1+\tau^2\omega^2$. Hence, if \eqref{the2} is satisfied, then
\begin{align}
D^2<1\leq1+\tau^2\omega^2=\left|L(i\omega)\right|^2,
\end{align}
for all $\omega \in \mathbb{R}$, which is a sufficient condition for stability
by \eqref{the2}. \qed 
\end{proof}

In Fig.\:\ref{fig1}\textbf{(a)}--\textbf{(c)}, we present the bifurcation diagrams 
showing the regions of stability and instability of the equilibrium solution
$v_0$. In the panels, the quantity 
$D:=|\beta|\int_{-\infty}^{\infty}|J(z)|dz$ is plotted against the bifurcation parameters $\alpha$, $r$, and $\tau$.  
\begin{figure}
\centering
\textbf{(a)}\includegraphics[width=7.5cm,height=6.0cm]{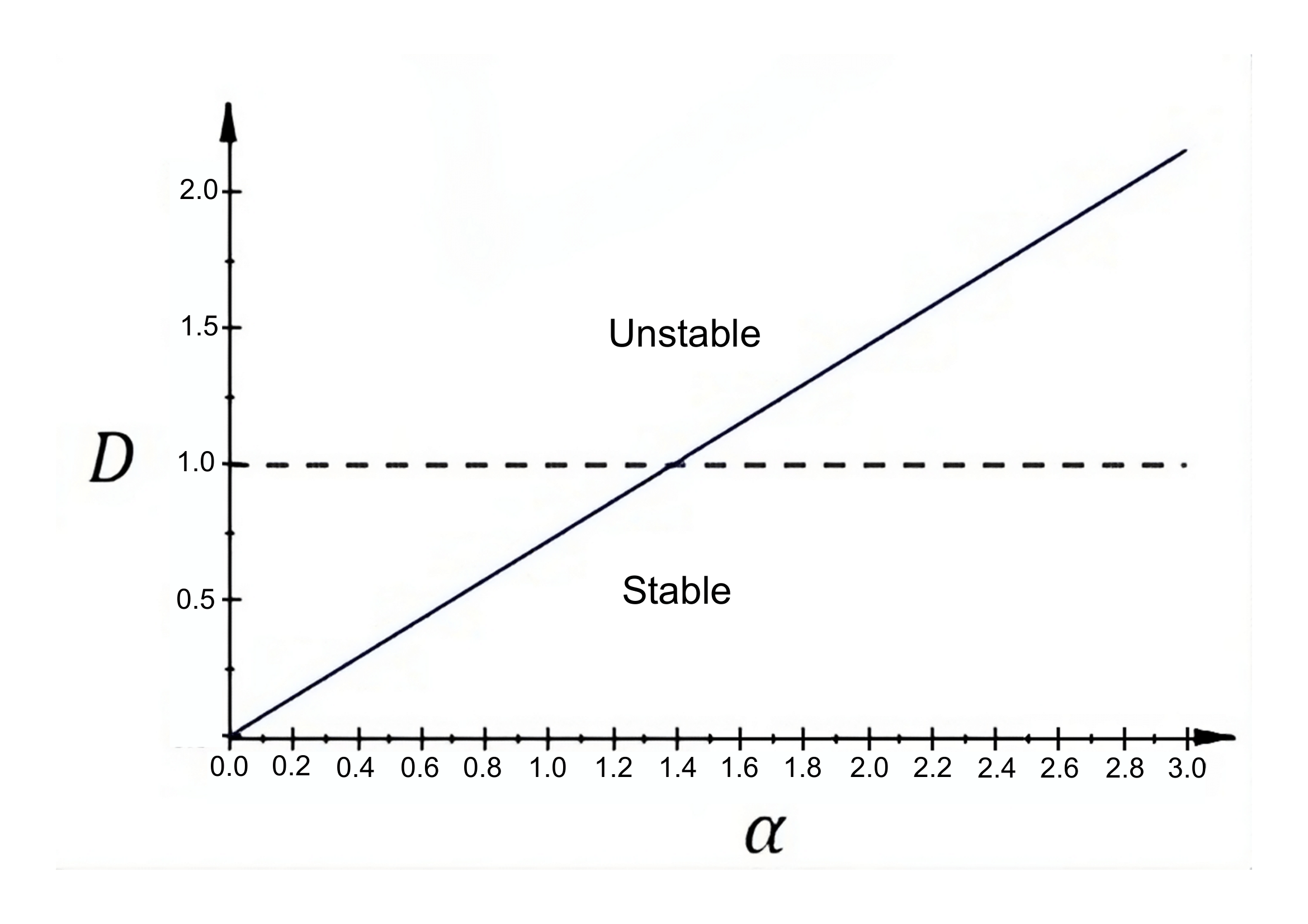}\textbf{(b)}\includegraphics[width=7.5cm,height=6.0cm]{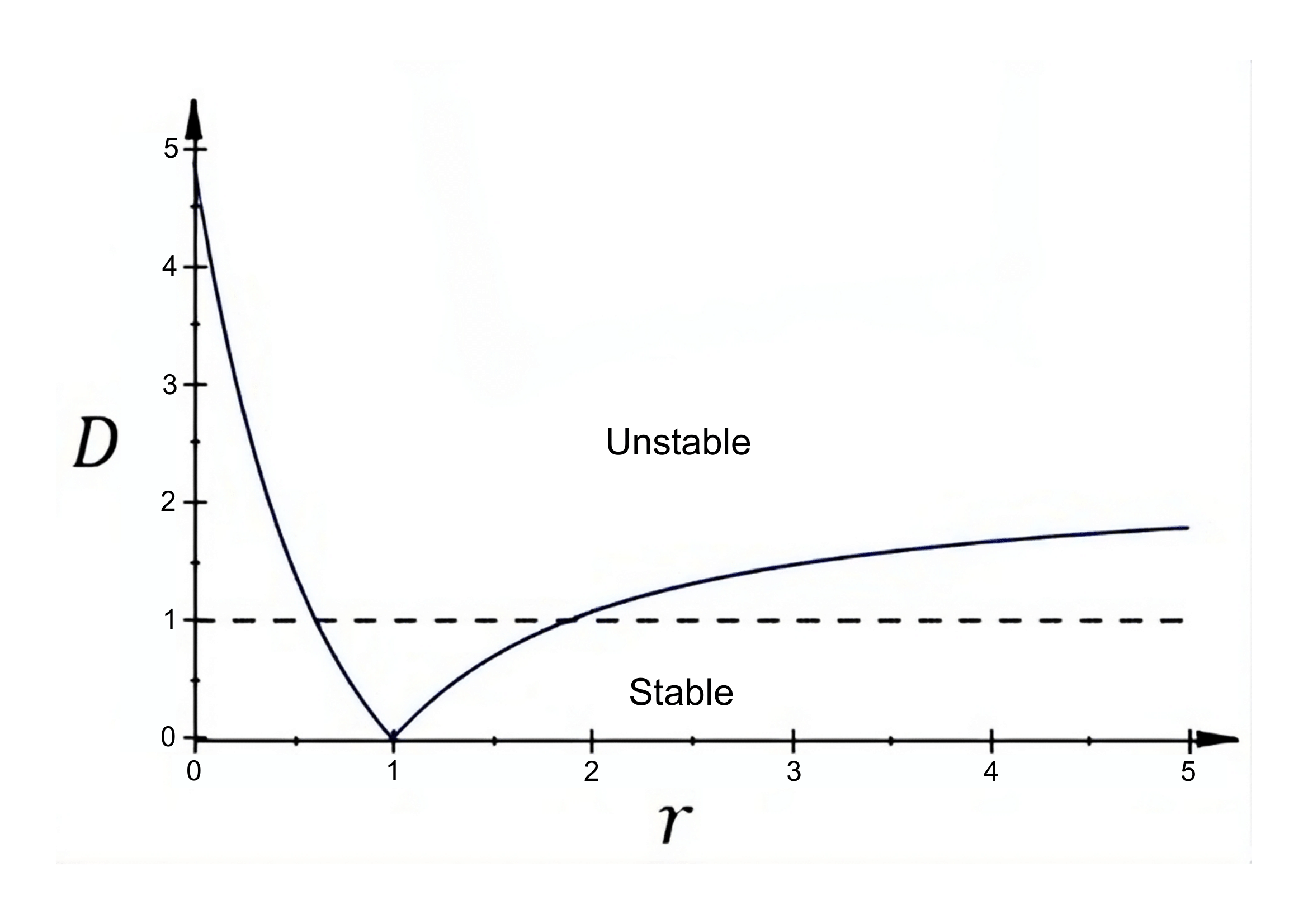}\\
\textbf{(c)}\includegraphics[width=7.5cm,height=6.0cm]{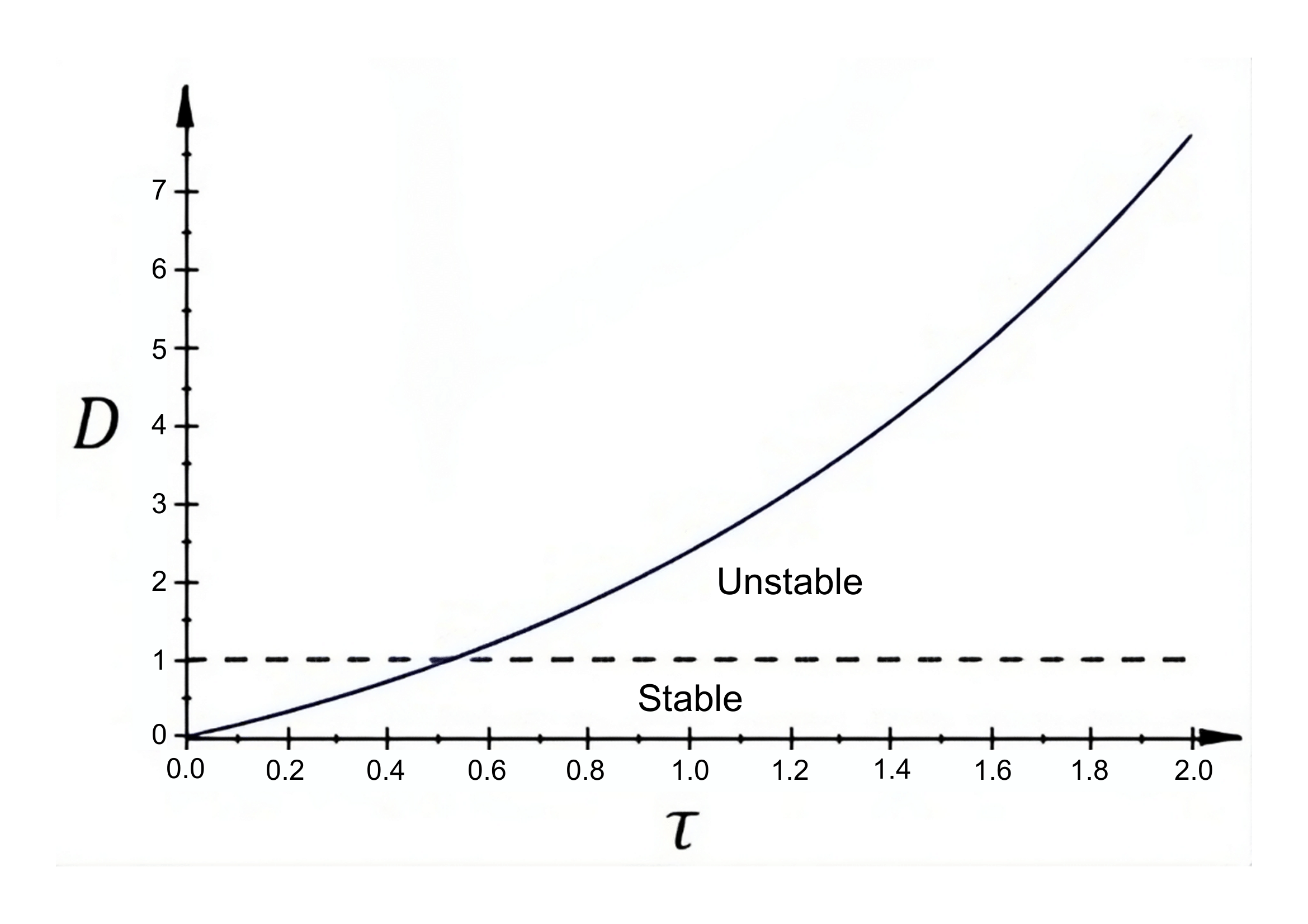}
\caption{The solid curves represent the quantity $D$ from Theorem \ref{theo1} plotted against the bifurcation parameters:
$\alpha$ in \textbf{(a)} with $\tau=0.7$, $r=0.5$; 
$r$ in \textbf{(b)} with $\alpha=2.0$, $\tau=0.7$; 
and $\tau$ in \textbf{(c)} with $\alpha=2.0$, $r=0.5$. 
The intervals of $\alpha\in(0.0,1.4)$, $r\in(0.60,1.84)$, and $\tau\in(0.0,0.51)$ in which the solid curves are
below the dashed horizontal line fulfills the sufficient condition of asymptotic stability of the  equilibrium solution 
$v_0 = \tau E$, following Theorem \ref{theo1}.
Other parameters are fixed at $c=15.0$, $E=0.275$, $a_{e} = 10.0$, $a_{i} =
a_{e}/r$. The nonlinear dependence on $\tau$, as opposed to the linear
  dependence on $\alpha$, arises because in $\beta=\alpha c \tau F'(v_0)$, we
  get an additional dependence since $v_0=\tau E$.} \label{fig1}
\end{figure}

We are interested in parameter constellations where spatiotemporal
patterns emerge, where the constant solution is \emph{not} stable. From the diagrams, this requires that the memory decay $\alpha$ and the time constant  $\tau$ of the synapses both be sufficiently large and the
  excitation and inhibition be sufficiently imbalanced. In fact, the
  case $r<1$, where inhibition is weaker, but more widely spread than
  excitation, is usually assumed in the literature to let $J$ acquire its
  Mexican hat shape. However, in this paper, we would use both cases: when $r<1$ (giving $J$ a Mexican-hat shape) and when $r>1$ (giving $J$ an inverse Mexican-hat shape.)


\section{Dynamic bifurcations of the equilibrium solution}\label{section4}
In Section \ref{section3}, we have seen that static bifurcations are
not possible since $\lambda=0$ is not a solution of \eqref{eq14a}. In this section, we investigate the conditions for
oscillatory (dynamic) bifurcations. In the case of a homogeneous neural field, we use the kernel function 
$J$ and the sigmoid transfer function $F$ in \eqref{num} and \eqref{num1}, 
respectively.  
An infinitesimal perturbation of the form $w(x,t) = e^{\lambda t} e^ {ikx}$ by \eqref{eq14a} would then need to be satisfied, and could be written as
\begin{align} \label{eq3.3}\nonumber
\tau\lambda^2 + (\tau\alpha+1)\lambda+\alpha&=\alpha c\tau F'(v_0)\lambda\int_{-\infty}^{\infty} J(z)e^{-\lambda\frac{|z|}{\nu}}e^{-ikz}dz\\
&=\beta\lambda\Bigg[\int_{-\infty}^{\infty}\Big(\frac{a_{e}}{2}e^{-|z|}-\frac{a_{i}r}{2}e^{-|z|r}\Big)e^{-\lambda\frac{|z|}{\nu}}e^{-ikz}dz\Bigg]\nonumber\\\nonumber
&=\beta\lambda\Bigg[a_{e}\frac{1+\frac{\lambda}{\nu}}{(1+\frac{\lambda}{\nu})^{2}+k^{2}}- a_{i}r\frac{r+\frac{\lambda}{\nu}}{(r+\frac{\lambda}{\nu})^{2}+k^{2}}\Bigg]\nonumber\\
&- i\beta\lambda\Bigg[\frac{a_e k}{(1+\frac{\lambda}{\nu})^{2}+k^{2}}-\frac{a_irk}{(r+\frac{\lambda}{\nu})^{2}+k^{2}}\Bigg].
\end{align}

In \eqref{eq3.3}, we shall consider the solution $\lambda$ as a function of $k$. The solution loses its stability when the real part of a root $\lambda$ in \eqref{eq3.3} changes from negative to positive. By tuning the parameters $r$, $\tau$ or $\alpha$, a critical point is eventually reached at $k=k_{c}$ in which the real part of the corresponding eigenvalue $\lambda(k_{c})$ of \eqref{eq3.3} become zero. 
From this critical point, one gets the critical wave-number
$k_{c}$ and the critical frequency  $\omega_c = \Im [\lambda(k_c)]$.
The case $k_{c}=0$ and $\omega_{c} \neq 0$ corresponds to a Hopf bifurcation 
\citep{folias2005breathers, laing2005spiral, folias2004breathing}, and the case $k_{c} \neq 0$ and $\omega_{c} \neq 0$  
to a Turing-Hopf bifurcation \citep{coombes2007modeling,
  venkov2007dynamic,touboul2012mean}. We shall 
investigate both cases in more detail.

\subsection{Hopf bifurcation}
Now, we insert $\lambda=i\omega$ in \eqref{eq3.3}, and because we are searching conditions for Hopf bifurcation (i.e., when $k_{c}=0$), we insert $k=k_c$ (so that the real part of the corresponding eigenvalue $\lambda|_{(k_{c}=0)}$ becomes
zero) to get a polynomial of degree six in $\omega$ given by 
\begin{equation} \label{eq3.4b1}
q_6\omega^6 + q_5\omega^5 + q_4\omega^4 + q_3\omega^3
+ q_2\omega^2 + q_1\omega=0,
\end{equation}
where the coefficients are given by 
\begin{equation} \label{eq3.4b2}
\begin{cases}
\displaystyle{q_6=\frac{\tau}{\nu^4}},\\[2.0mm] \displaystyle{q_5=\frac{\tau\alpha + 1}{\nu^4}},\\[2.0mm] \displaystyle{q_4=\frac{\tau(1+r^2)}{\nu^2}+\frac{\beta(a_e - a_ir)}{\nu^3}-\frac{\alpha}{\nu^4}},\\[2.0mm] \displaystyle{q_3=\frac{(r^2+1)(\tau\alpha +1)+\beta(a_ir-a_e)}{\nu^2}},\\[2.0mm]
\displaystyle{q_2=\tau r^2+\frac{\beta r (a_er-a_i)}{\nu} - \frac{\alpha(1+r^2)}{\nu^2}},\\[2.0mm] \displaystyle{q_1=r^2(\tau\alpha+1)+\beta r(a_i-a_er)}.
\end{cases}
\end{equation}
A trivial solution of \eqref{eq3.4b1} is $\omega=0$, but $\lambda$ should be purely imaginary, i.e., $\omega\neq0$. However, this trivial solution allows us to reduce the degree of \eqref{eq3.4b1} to get
\begin{equation} \label{eq3.4b2}
q_6\omega^5 + q_5\omega^4 + q_4\omega^3 + q_3\omega^2
+ q_2\omega + q_1=0.
\end{equation}
Substituting a solution $i\omega_{c}$ ($\omega_c\neq0$) in \eqref{eq3.4b2} and
separating the real and imaginary parts yields
\begin{equation} \label{eq3.4b3}
  \begin{cases}
   a_3 \omega_{c}^
{4}+ a_2 \omega_{c}^{2}+a_1=0
,\\
b_3 \omega_{c}^{4}+ b_2\omega_{c}^{2}+ b_1=0,
  \end{cases}  
\end{equation}
where 
\begin{equation} \label{eq3.4b4}
  \begin{cases}
\displaystyle{a_3=  \frac{1+\alpha\tau}{\nu^4},}\\[2.0mm]
\displaystyle{a_2=-\frac{\beta(a_ir-a_e) + (\alpha\tau+1)(r^2+1)}{\nu^2},}\\[2.0mm]
\displaystyle{a_1=\beta r(a_i-a_er)+ r^2(\alpha\tau+1),}\\[2.0mm]
\displaystyle{b_3=\frac{\tau}{\nu^4},}\\[2.0mm]
\displaystyle{b_2= {\frac {\alpha}{{\nu}^
{4}}}-\frac{\beta(a_e-a_ir)}{\nu^3}-\frac{\tau(1+r^2)}{\nu^2}},\\[2.0mm]
\displaystyle{b_1={r}^{2}\tau+\frac{\beta r(a_er-a_i)}{\nu}}-\frac{\alpha(1+r^2)}{\nu^2}.
  \end{cases}  
\end{equation}
From the first equation of \eqref{eq3.4b3}, we get
\begin{equation}\label{eq3.4b5}
\omega_{c}^{2}={\frac {-a_{2}\pm\sqrt {{a_{2}}^{2}-4\,a_{1}\,a_{3}}
}{2a_{3}}}.
\end{equation}
Since $\omega_{c}\in\mathbb{R}-\{0\}$, $\omega_{c}^{2}>0$. And because $a_3>0$, \eqref{eq3.4b5} can be satisfied only if 
\begin{equation}\label{eq3.4b7}
\begin{cases}
a_ir-a_e\geq0,\\
\Gamma=\displaystyle{\big[\beta(a_ir-a_e) + (\alpha\tau+1)(r^2+1)\big]^2-
4\big[1+\alpha\tau\big]
\big[\beta r(a_i-a_er)+ r^2(\alpha\tau+1)\big]\geq0}.
\end{cases}
\end{equation}
  
When we substitute  $\omega_{c}^{2}=({-a_{2}+\sqrt {{a_{2}}^{2}-4\,a_{1}\,a_{3}}})/2a_{3}$
into the second equation of \eqref{eq3.4b3}, we obtain 
 \begin{align}\label{eq3.4b8}
  \Big[\tau\, \Delta ^{2}-\tau\, \left( {r}^{2}+1 \right) 
  \Delta+{r}^{2}\tau\Big]\nu^2+\Big[ \beta\, \left( a_{i}\,r-a_{e} \right) 
 \Delta +\beta\,r \left( 
a_{e}\,r-a_{i}\right) \Big]\nu+\Big[ \alpha\,  \Delta  -\alpha
\, \left( {r}^{2}+1 \right)\Big]=0,
 \end{align}
 where
\begin{align} \label{eq3.4b9}
& \Delta=\frac{1}{2\alpha\tau+2}\Bigg[\sqrt {-4\, \left( \alpha\,\tau+1 \right)  \left[ \beta\,r \left( -a_{
e}\,r+a_{i} \right) + \left( \alpha\,\tau+1 \right) {r}^{2} \right] +
 \left[ \beta\, \left( a_{i}\,r-a_{e} \right) +
\left( \alpha\,\tau+1
 \right)  \left( {r}^{2}+1 \right)  \right] ^{2}}
  \nonumber\\&\:\:\:\:+\beta\, \left( a_{i}\,r-a_{e} \right) + \left( \alpha\,\tau+1 \right) 
 \left( {r}^{2}+1 \right)\Bigg].
\end{align}
The relations in \eqref{eq3.4b7} provide the parametric region in which
the neural field oscillates due to Hopf bifurcation that occurs at $\omega_{c}^{2}=({-a_{2}+\sqrt {{a_{2}}^{2}-4\,a_{1}\,a_{3}}})/2a_{3}$ in \eqref{eq3.4b5} and which must also satisfy the first equation in \eqref{eq3.4b3}. It is worth noting that in \eqref{eq3.4b7}, the first condition is always satisfied (since in our analysis, we always have $r=a_e/a_i$). Hence, it suffices to find the parametric regions where the second condition in \eqref{eq3.4b3} (i.e., $\Gamma\ge0$) holds, for each parameter of interest. Fig.\:\ref{hopf4a}\textbf{(a)} and \textbf{(b)} show the Hopf bifurcation curves for the parameters $\alpha$ and $\tau$, respectively. Fig.\:\ref{hopf4a} \textbf{(c)}-\textbf{(e)} show the Hopf bifurcation curves (in blue) in the parameter spaces $\alpha$-$\nu$, $\tau$-$\nu$, and $r$-$\nu$, respectively, in which the gray regions satisfy \eqref{eq3.4b7}.

Fig.\:\ref{hopf4aa} shows corresponding space-time patterns of the membrane potential oscillating in different regions of the Hopf bifurcation parameter space of Fig.\:\ref{hopf4a}\textbf{(c)}.  In Fig.\:\ref{hopf4aa}\textbf{(a)}-\textbf(c), the values of the parameters $\nu$ and $\alpha$ are chosen below, on, and above the Hopf bifurcation curve of
Fig.\:\ref{hopf4a}\textbf{(c)}, leading to oscillatory regimes with different frequencies and amplitudes.

\begin{figure}[H]
\begin{center}
\textbf{(a)}\includegraphics[width=7.5cm,height=4.5cm]{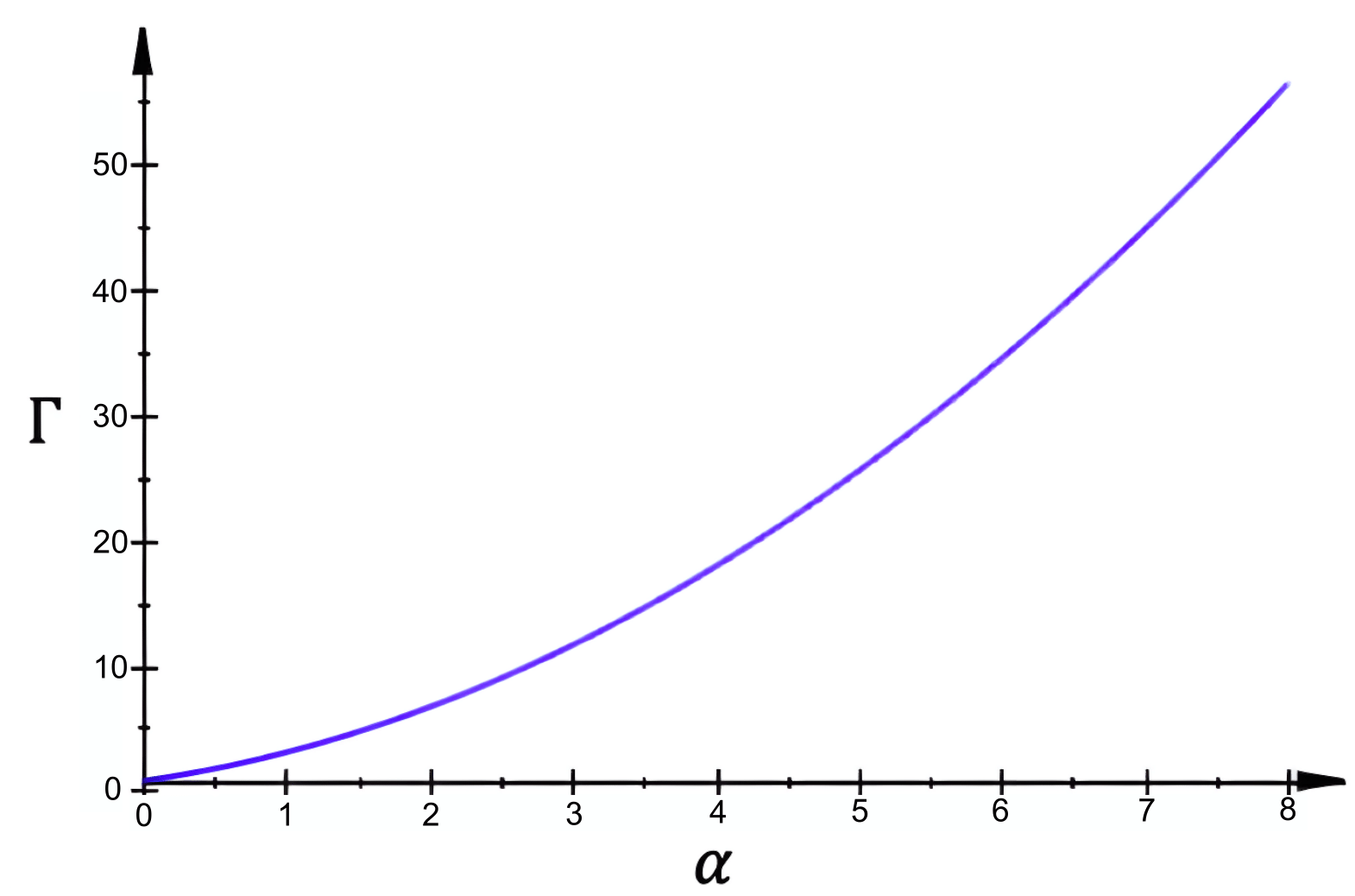}\textbf{(b)}\includegraphics[width=7.5cm,height=4.5cm]{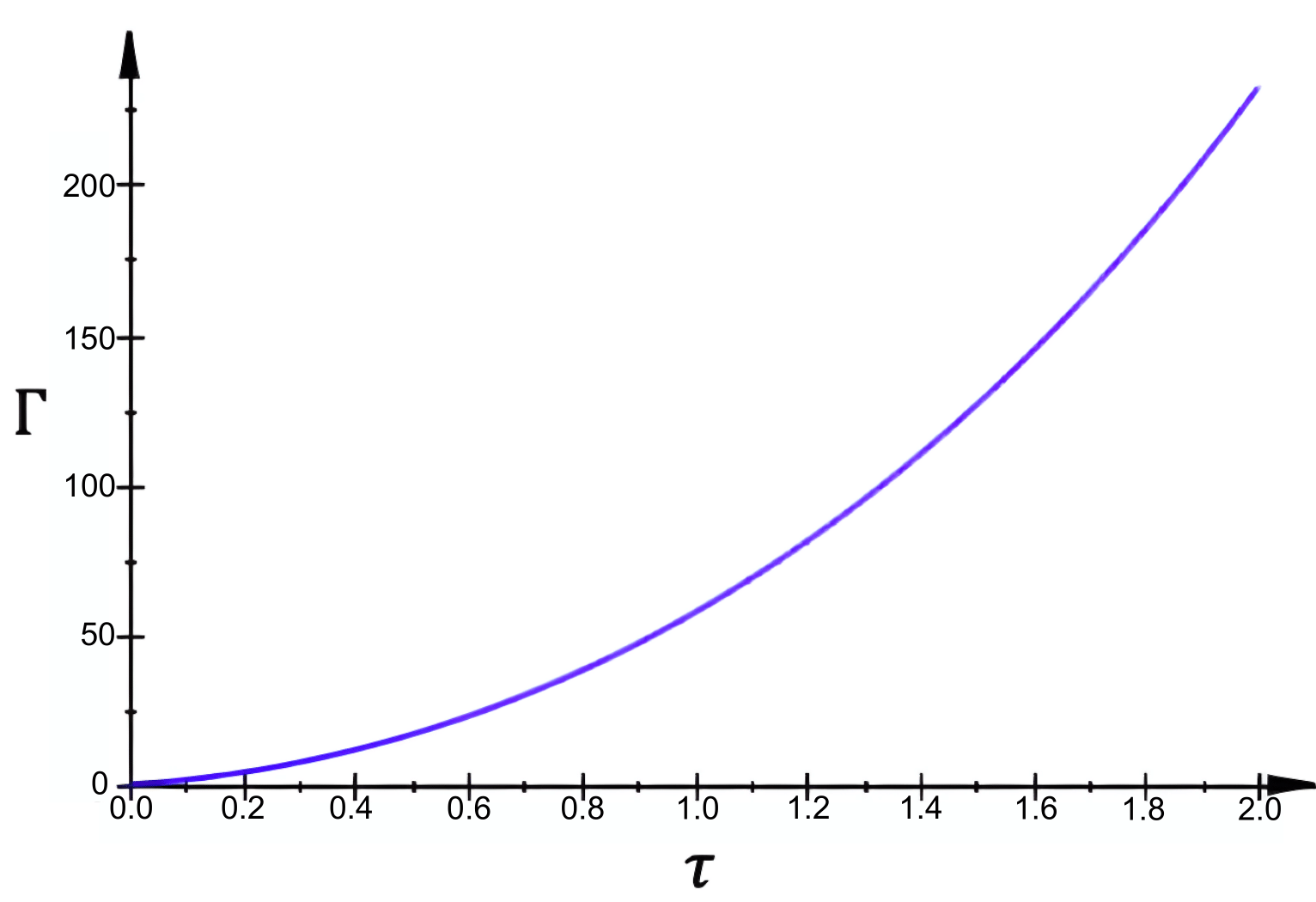}
\textbf{(c)}\includegraphics[width=7.5cm,height=4.5cm]{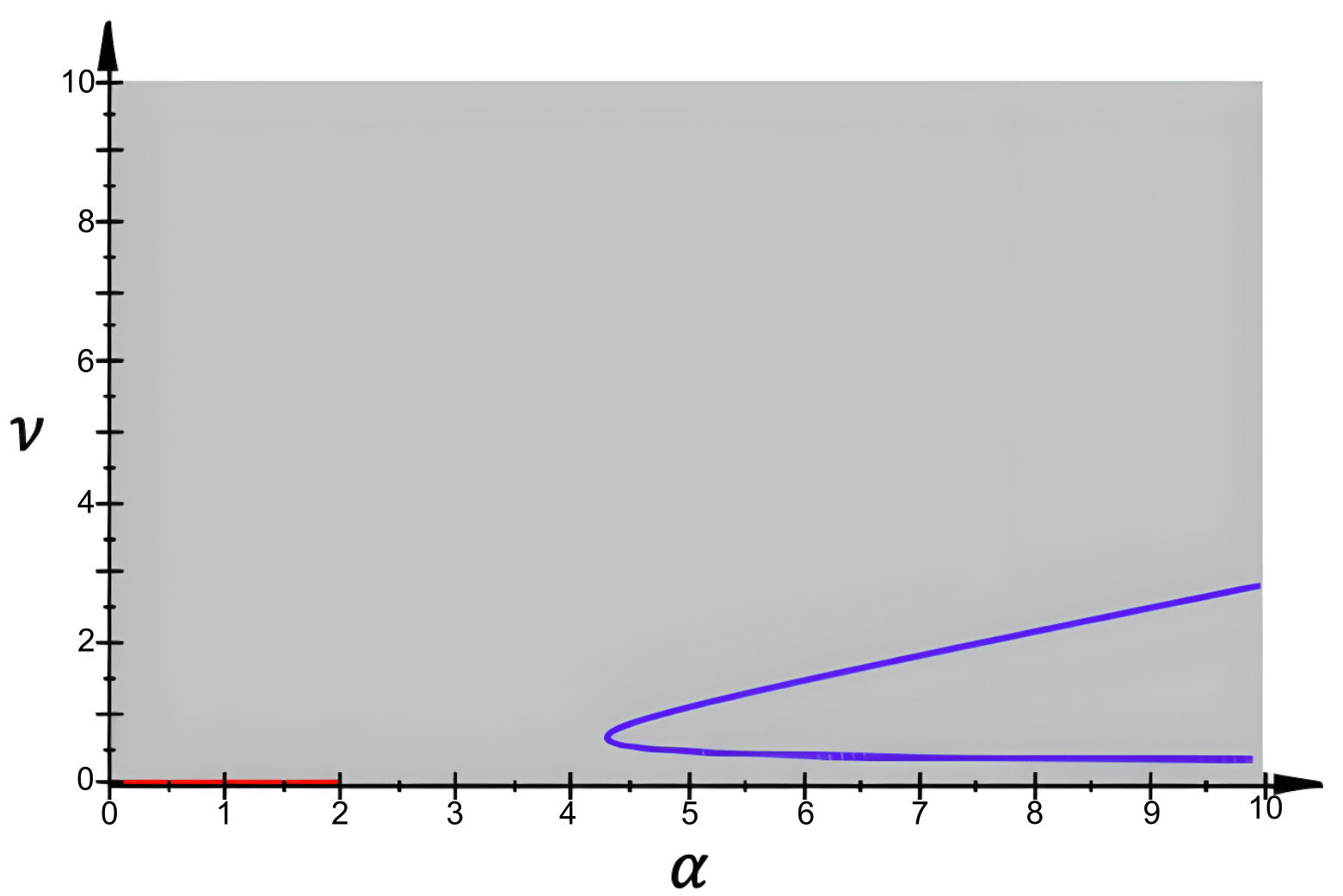}\textbf{(d)}\includegraphics[width=7.5cm,height=4.5cm]{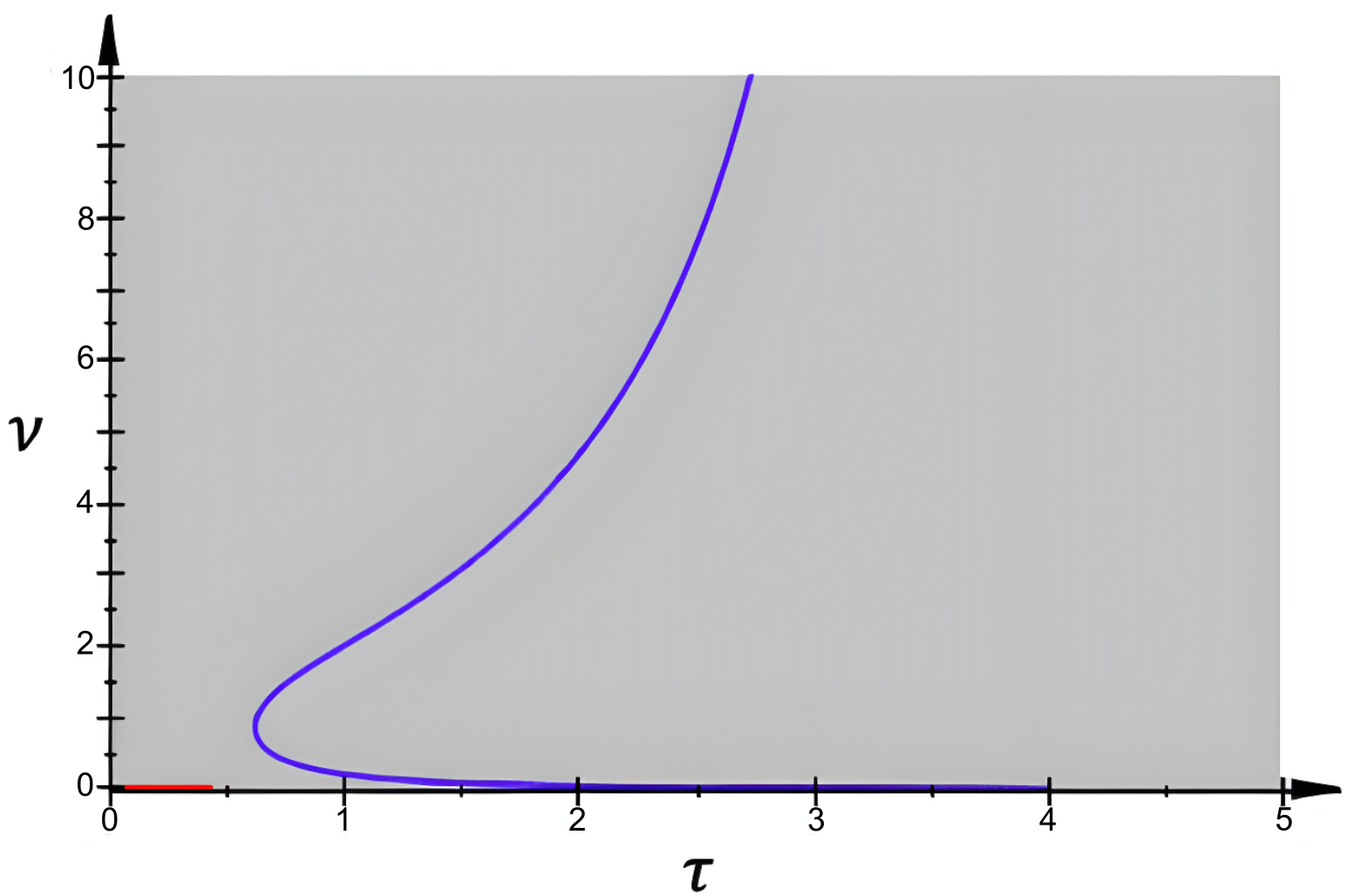}
\textbf{(e)}\includegraphics[width=7.5cm,height=4.5cm]{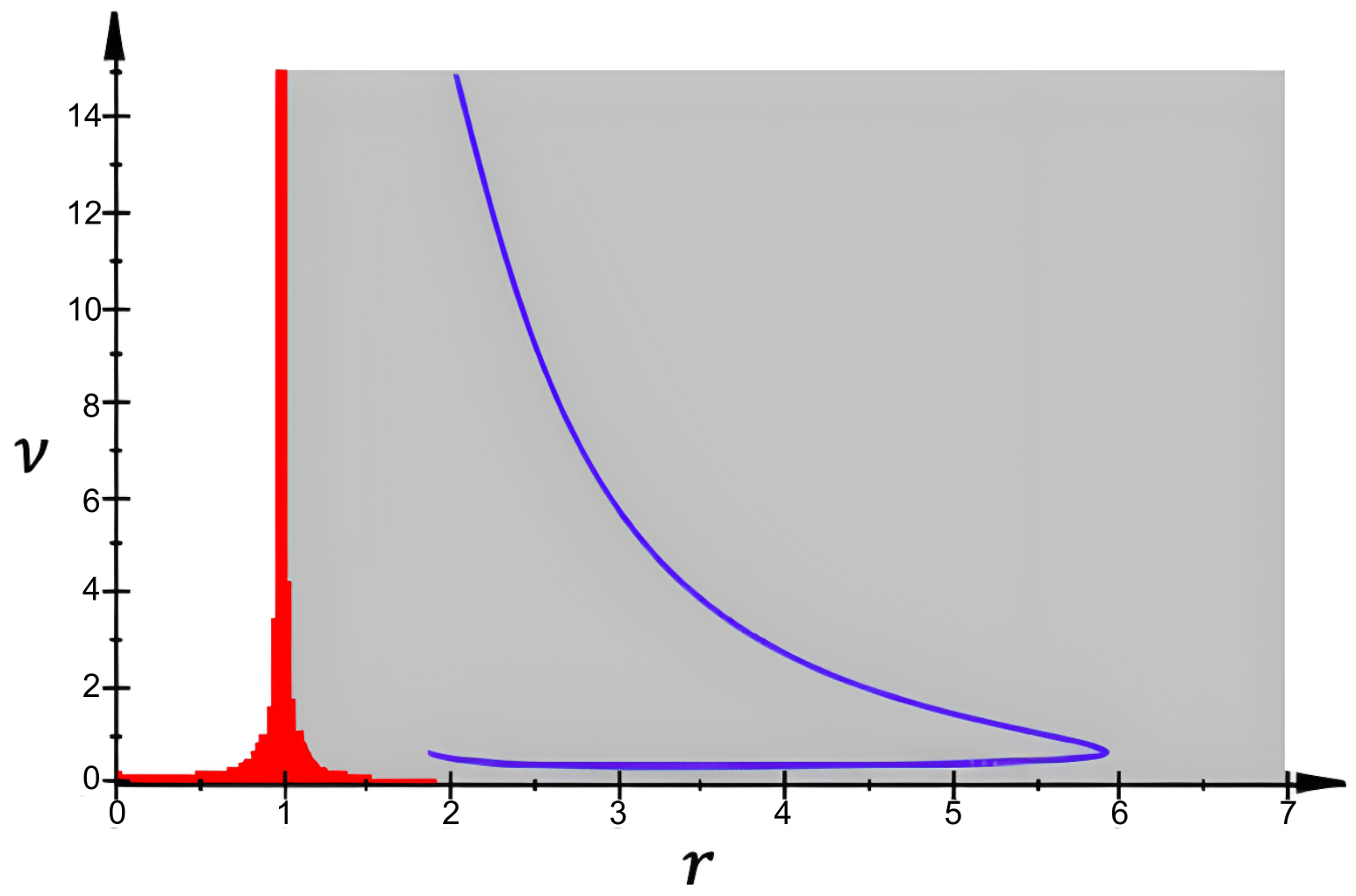}
\caption{ Panels \textbf{(a)} and \textbf{(b)} show the interval of 
$\alpha\in(0.0,8.0)$ and $\tau\in(0.0,2.0)$ for which  \eqref{eq3.4b7} is satisfied.
In the panels \textbf{(c)}-\textbf{(e)}, the blue curves represent the solutions of \eqref{eq3.4b8} 
in the parameter spaces $\alpha$-$\nu$, $\tau$-$\nu$, and $r$-$\nu$, respectively.
The gray areas in the panels represent the region of the parameter spaces where \eqref{eq3.4b7} holds, i.e., the oscillatory region. The parts of blue curves from \eqref{eq3.4b8} that lie in the gray region represent the values of the parameters for which oscillations exist, as these values satisfy  \eqref{eq3.4b7}.
In \textbf{(a)} $\tau=0.75$, $r=5.0$;
in \textbf{(b)} $\alpha=6.0$, $r=5.0$;
in \textbf{(c)} $\tau=0.75$, $r=5.0$;
in \textbf{(d)} $\alpha=6.0$, $r=5.0$; and
in \textbf{(e)} $\alpha=6.0$, $\tau=0.75$. 
In \textbf{(a)}-\textbf{(e)}, the other parameter values are: 
$c=15.0$, $E=0.275$, $a_{e}=10.0$, $v_0=\tau E$, $\beta=\alpha c\tau F'(v_0)$, and  in \textbf{(e)} $a_{i}=a_e/r$.} \label{hopf4a}
\end{center}
\end{figure}

\begin{figure}[H]
\begin{center}
\textbf{(a)}\includegraphics[width=7.5cm,height=4.5cm]{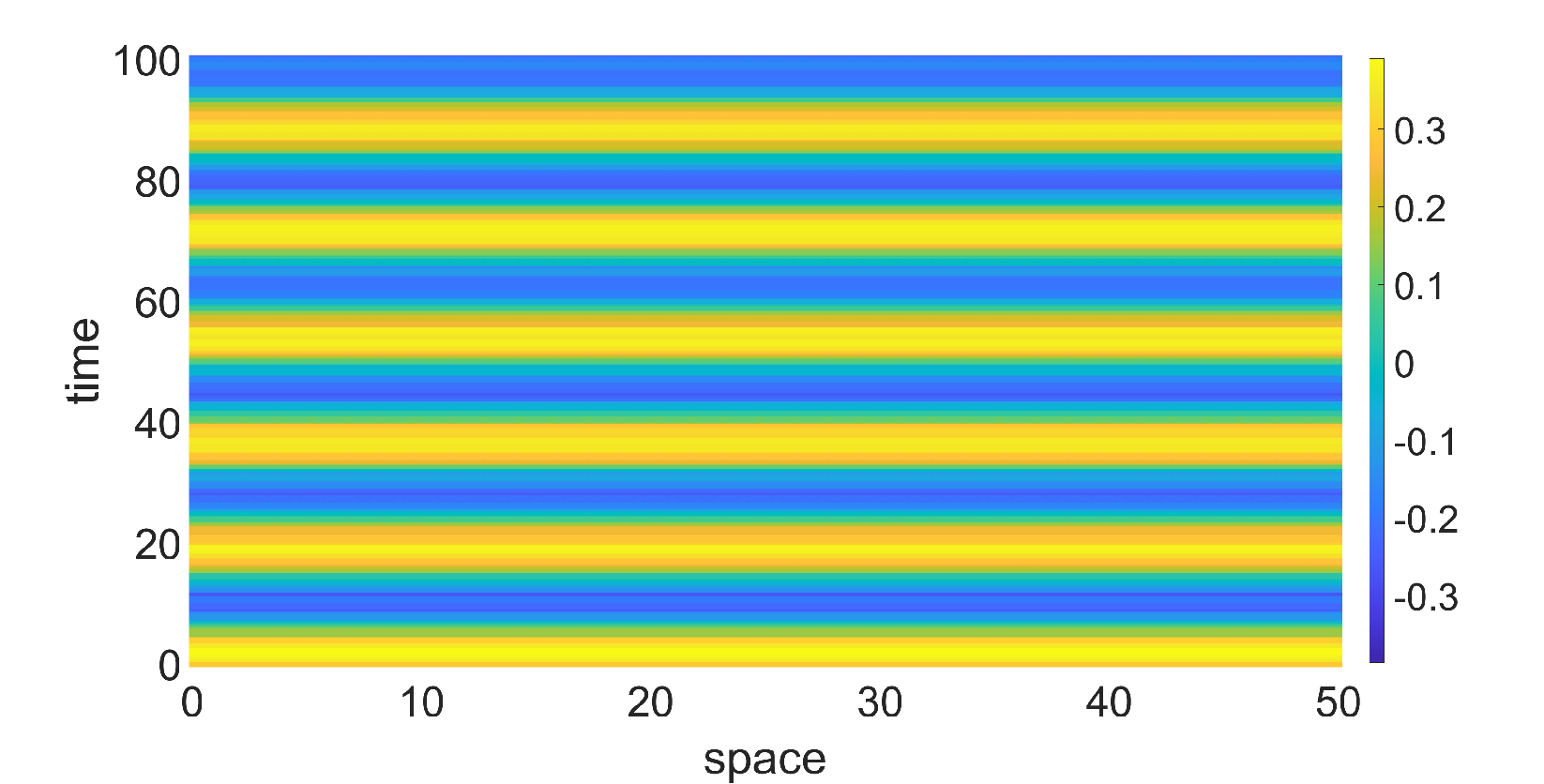}\textbf{(b)}\includegraphics[width=7.5cm,height=4.5cm]{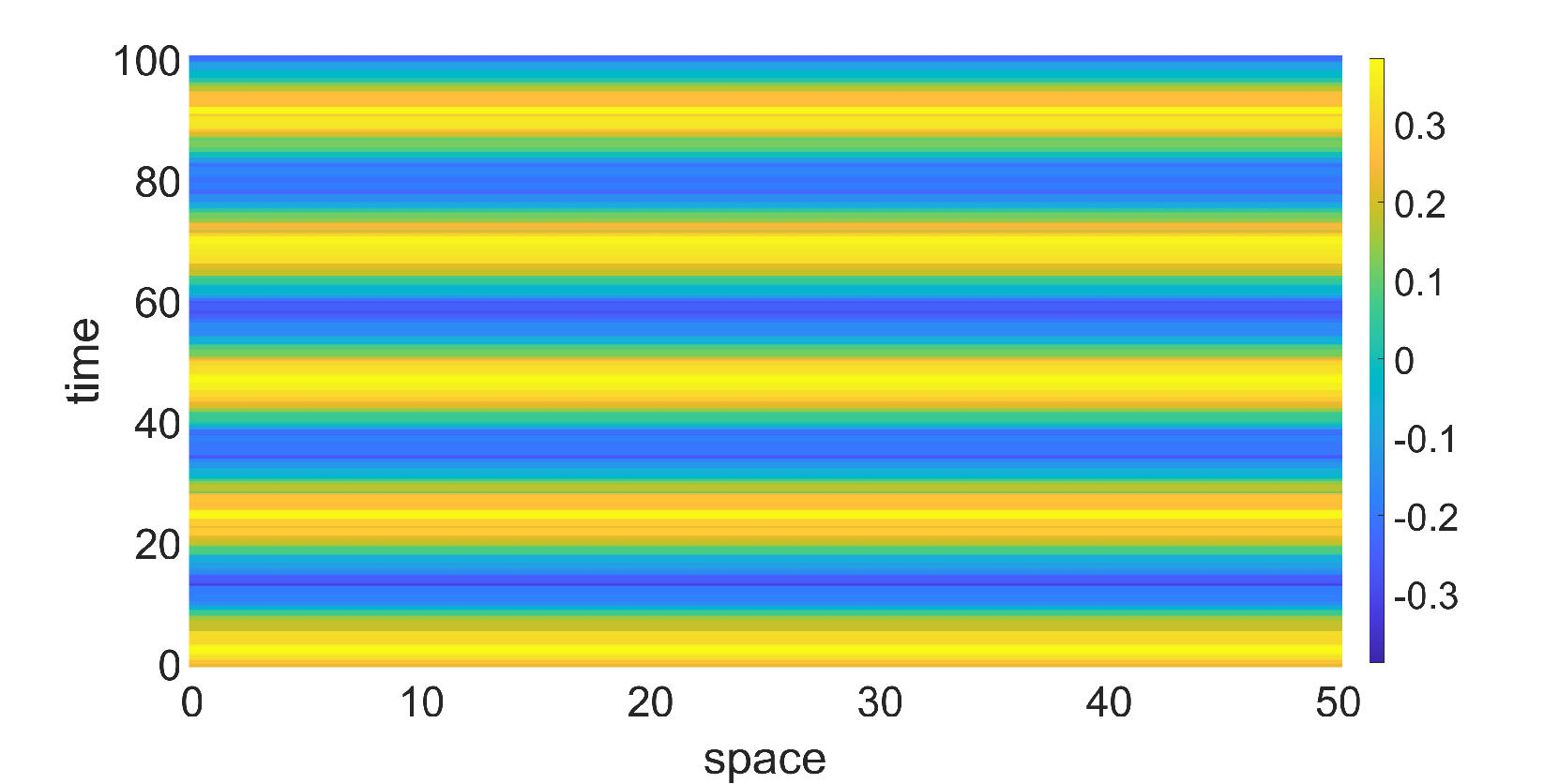}\textbf{(c)}
\includegraphics[width=7.5cm,height=4.5cm]{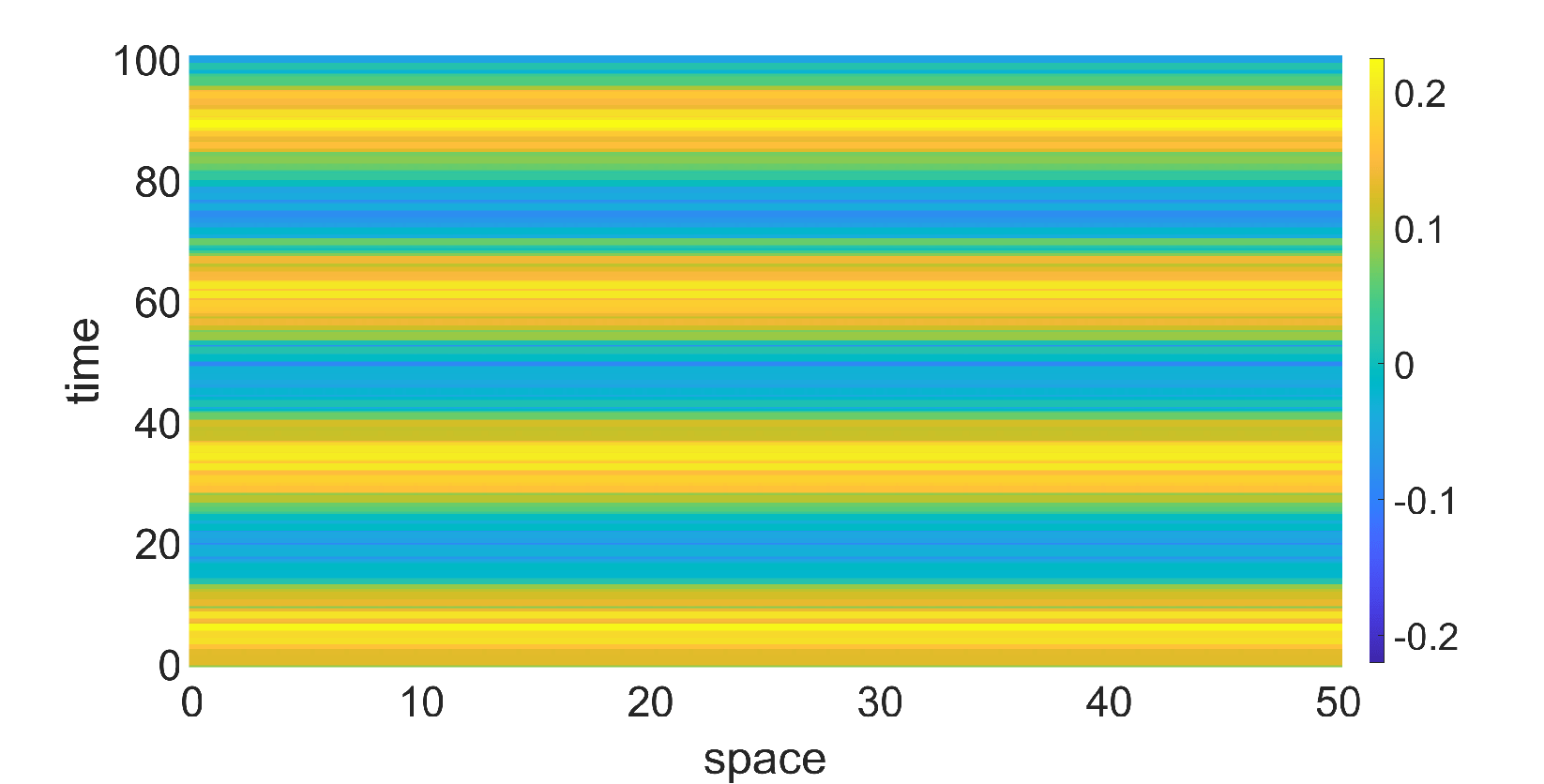}
\caption{Panels \textbf{(a)}--\textbf{(c)} show color coded space-time patterns of the membrane potential $v(x,t)$ emerging from Hopf instability.
In \textbf{(a)}, $\alpha=7.0$, $\nu=0.16$, i.e., below the Hopf bifurcation (blue) curve. 
In \textbf{(b)}, $\alpha=7.0$, $\nu=1.83$, i.e., on the Hopf bifurcation curve.
In \textbf{(c)}, $\alpha=7.0$, $\nu=2.5$, i.e., above the Hopf bifurcation curve.
In all cases, we have periodic oscillations of spatially constant solutions. The solutions are obtained for the Gaussian connectivity kernel in the panels, with initial conditions chosen randomly from a uniform distribution on $[v_0-0.1, v_0+0.1]$.
Other parameters are fixed at $\tau=0.75$, $a_e=10.0$, $a_i=2.0$, $r=a_e/a_i=5.0$, $c=15.0$, and $k=0$.} \label{hopf4aa}
\end{center}
\end{figure}
\subsection{Turing-Hopf bifurcation}
\begin{theorem} \label{theo3}
Let $D := |\beta|\int_{-\infty}^{\infty}\left | J(z) \right |dz$ and $L(\lambda)$ be a polynomial whose roots have non-positive real parts. Then there exists $B> 0$ depending only on $L$ and $D$ such that
\begin{align} \label{theo3.1}
\vert\omega\vert \leq B,
\end{align}
whenever $w(x,t)=e^{i\omega t}e^{ikx}$, $\omega,k\in  \mathbb{R}$, is a solution of \eqref{eq120}. Furthermore, if $D<1$, then there exists $A>0$, depending on $L$ and $D$, such that 
\begin{align} \label{theo3.2}
0<A\leq \vert \omega \vert.
\end{align}
In particular, if $L(\lambda)=\tau\lambda+1$, then
\begin{align} \label{theo3.3}
\tau^{2}\omega^{2}\leq D^{2}-1.
\end{align}
\end{theorem}
\begin{proof}
If $\lambda=i\omega$ satisfies the dispersion relation \eqref{eq14a} for some $k$, then
\begin{align} \label{pr3.1}
\vert L(i\omega) \vert \leq \beta \int_{-\infty}^{\infty}\vert J(z)\vert dz=D.
\end{align}
Since $\vert L(i\omega) \vert\rightarrow \infty$ as $\omega\rightarrow\pm \infty$ for any non-constant polynomial $L$, the above inequality implies an upper bound $B$ on $\vert \omega \vert$, which proves \eqref{theo3.1}. For the particular case $L(\lambda)=\tau\lambda+1$, \eqref{pr3.1} gives
\begin{align}
\vert L(i\omega)\vert^{2}=\tau^{2}\omega^{2}+1\leq D^{2},
\end{align}
which proves \eqref{theo3.3}.
\end{proof}

The following analytical result will be used in the rest of the numerical computations. The conditions for Turing-Hopf bifurcation require that with the 
 Fourier-Laplace ansatz \eqref{ansatz}, that is, 
 $w(x,t)=e^{\lambda t}e^{ikx}$, we find
$\lambda=\pm i\omega$ with $\omega \neq 0$ at some critical value $k_c\neq 0$.
Inserting $\lambda=i\omega$ in \eqref{eq14a}, we obtain
\begin{align} \label{eq4.1}
L(i\omega):=1+i\tau\omega=\beta
\Big(\frac{i\omega}{\alpha+i\omega}\Big)\int_{-\infty}^{\infty}J(z)e^{-i\omega\frac{|z|}{\nu}}e^{ikz}dz,
\end{align}
which yields upon expansion,
\begin{align} \label{eq4.2}
\beta \frac{\omega^{2}+i\alpha\omega}{\alpha^{2}+\omega^{2}}\int_{-\infty}^{\infty}J(z)
e^{-\frac{i\omega\left | z \right |}{\nu}}\cos(kz)dz &=\beta
\frac{\omega^{2}+i\alpha\omega}{\alpha^{2}+\omega^{2}}\int_{-\infty}^{\infty}J(z)\Big[\cos\Big(\frac{\omega\left
| z \right |}{\nu}\Big)
-i\sin\Big(\frac{\omega\left | z \right |}{\nu}\Big)\Big]\cos(kz)dz\nonumber\\
&=\frac{\beta}{2} \frac{\omega^{2}+i\alpha\omega}{\alpha^{2}+\omega^{2}}\int_{-\infty}^{\infty}J(z)\left[
e^{i\left | z \right |(\frac{\omega}{\nu}+k)}+e^{i\left | z \right
|(\frac{\omega}{\nu}-k)}\right] dz.
\end{align}
By substituting the power series 
\begin{align*}
e^{i\left | z \right |(\frac{\omega}{\nu}\pm k)}=\sum_{m=0}^{N }\frac{i^{m}(\frac{\omega }{\nu }\pm k)^{m}}{m!}\left | z \right |^{m}+O(\nu^{-(N+1)}),
\end{align*}
\eqref{eq4.2} is written as
\begin{align} \label{eq4.3}
\frac{\beta}{2} \frac{\omega^{2}+i\alpha\omega}{\alpha^{2}+\omega^{2}}\int_{-\infty}^{\infty}J(z)
\Big(\sum_{m=0}^{N }\frac{i^{m}}{m!}\left[ \Big(\frac{\omega }{\nu
}+k\Big)^{m}+\Big(\frac{\omega }{\nu }-k\Big)^{m}\right] \left | z
\right |^{m}+O(\nu^{-(N+1)})\Big)dz.
\end{align}
We define $J_{m}$ as 
\begin{align} \label{eq20}
J_{m}:=\int_{-\infty }^{\infty }J(z)\left | z \right |^{m}dz,
\end{align}
and the integrals are assumed to exist. Substituting \eqref{eq20} into \eqref{eq4.3} yields
\begin{align} \label{eq4.4}
L(i\omega)=\beta \frac{\omega^{2}+i\alpha\omega}{\alpha^{2}+\omega^{2}}\left[J_{0}+i\frac{\omega}{\nu}J_{1}
-\frac{1}{2!}\Big(k^{2}+\frac{\omega^{2}}{\nu^{2}}\Big)J_{2}
+\cdots \right].
\end{align}
Equating the right-hand side of \eqref{eq4.4} to the left-hand side of \eqref{eq4.1}, we get
\begin{align} \label{eq21}
1+i\tau\omega=\beta \frac{\omega^{2}+i\alpha\omega}{\alpha^{2}+\omega^{2}}\left[J_{0}+i\frac{\omega}{\nu}J_{1}
-\frac{1}{2!}\Big(k^{2}+\frac{\omega^{2}}{\nu^{2}}\Big)J_{2}
+\cdots \right].
\end{align}
The number of terms required for the above series to be helpful depends on the values of $\nu$ and $k$ as well as the shape of the kernel $J$. If $J$ is highly concentrated near the origin, as in our case (see \eqref{num}) or, more generally, if $J$ is of exponential order, then a few terms will suffice. That is, assume there exist positive numbers $\kappa_{1}$ and $\kappa_{2}$ such that
\begin{equation}
\left | J(z) \right |\leq \kappa_{1}e^{-\kappa_{2}\left | z \right |} \quad\text{ for all $z\in \mathbb{R}$}.
\end{equation}
Then by \eqref{eq20}
\begin{align*}
\left | J_{m} \right |\leq \int_{-\infty }^{\infty }\left | z \right |^{m}\kappa_{1}e^{-\kappa_{2}\left | z \right |}dz=2\kappa_{1}\int_{0}^{\infty}z^{m}e^{-\kappa_{2}z}dz=2\kappa_{1}\kappa_{2}^{-(m+1)}\Gamma(m+1)=2\kappa_{1}\kappa_{2}^{-(m+1)}m!,
\end{align*}
so the $m$th term in the series \eqref{eq21} is bounded in absolute value by
\begin{align*}
2\frac{\kappa_{1}}{\kappa_{2}}\Big(\frac{\left | \omega \right |}{\kappa_{2}\nu}\Big)^{m}\leq 2\frac{\kappa_{1}}{\kappa_{2}}\Big(\frac{B}{\kappa_{2}\nu}\Big)^{m},
\end{align*}
where we have used Theorem \ref{theo3} to bound the values of $\omega$. In the case of a high transmission speed $\nu$ or $B$ (for example, small $\beta$) or a large value of $\kappa_{2}$ (rapid decrease of $J$ away from the origin) or a bounded value of $k$, the finite series has increased precision. At least one of these conditions is assumed to be true, so a few terms are sufficient to determine the general behavior. To observe the qualitative effects of a finite transmission speed, we, therefore, neglect the terms from the fourth and higher orders in the series \eqref{eq21}.

Equating the real parts of both sides in \eqref{eq21}, and similarly with the imaginary parts, and considering $\omega \neq 0$
(a Turing-Hopf bifurcation condition), we have
\begin{align}\label{eq4.5}
\begin{cases} 
\begin{split}
1=\frac{\beta}{\alpha^{2}+\omega^{2}}\left[ \omega^2 J_{0}-\alpha\omega
\left( \frac{\omega}{\nu}J_{1} \right)-\omega^{2}\frac{1}{2!}(k^{2}+\frac{\omega^{2}}{\nu^{2}})J_{2} \right],\\[2.5mm]
\tau=\frac{\beta}{\alpha^{2}+\omega^{2}}\left[ \alpha
J_{0}+\omega\left(\frac{\omega}{\nu}J_{1} \right)
-\frac{1}{2!}\alpha
\Big(\big(k^{2}+\frac{\omega^{2}}{\nu^{2}}\big)J_{2}
\Big) \right].
\end{split}
\end{cases}
\end{align}
From \eqref{eq4.5} we have
\begin{align} \label{eq.4.12}
\frac{\tau}{\alpha
J_{0}+\frac{\omega^{2}}{\nu}J_{1}-\frac{\alpha}{2}\Big(k^{2}+\frac{\omega^{2}}{\nu^{2}}\Big)J_{2}}
=\frac{\beta}{\alpha^{2}+\omega^{2}}=\frac{1}{\omega^{2}J_{0}-\alpha\frac{\omega^{2}}{\nu}J_{1}-\frac{\omega^{2}}{2}\Big(k^{2}+\frac{\omega^{2}}{\nu^{2}}
\Big)J_{2}},
\end{align}
which gives
\begin{align} \label{eq4.15}
\alpha J_{0}+\frac{\omega^{2}}{\nu}J_{1}-\frac{\alpha}{2}\Big(k^{2}+\frac{\omega^{2}}{\nu^{2}}\Big)J_{2}
=\tau\omega^{2}J_{0}-\tau\alpha\frac{\omega^{2}}{\nu}J_{1}-\tau\frac{\omega^{2}}{2}\Big(k^{2}+\frac{\omega^{2}}{\nu^{2}}\Big)J_{2}.
\end{align}
After substituting \eqref{num} into \eqref{eq20}, the convergent improper integrals $J_{n}$ $(n=0,1,2)$ are 	explicitly calculated as
\begin{align} \label{eq4.1800}
\begin{split}
&J_{0}=-a_{i}+a_{e},\\[1.0mm]
&J_{1}=\frac{ra_{e}-a_{i}}{r},\\[1.0mm]
&J_{2}=\frac{2(r^{2}a_{e}-a_{i})}{r^{2}}.
\end{split}
\end{align}
Substituting \eqref{eq4.1800} into \eqref{eq4.15}, we obtain
\begin{align} \label{Turing1}
&\frac{\tau}{\nu^2}\Bigg(\frac{r^2a_e-a_i}{r^2}\Bigg)\omega^4+\Bigg[
\tau(a_i-a_e)+\Big(\tau k^2-\frac{\alpha}{\nu^2}\Big)\Big(\frac{r^2a_e-a_i}{r^2}\Big)+\Big(\frac{\alpha\tau+1}{\nu}\Big)\Big(\frac{ra_e-a_i}{r}\Big)\Bigg]\omega^2\nonumber\\
&+\alpha\Bigg((a_e-a_i)-k^2\Big(\frac{r^2a_e-a_i}{r^2}\Big)\Bigg)=0.
\end{align}
Thus, \eqref{Turing1} represents the Turing-Hopf bifurcation in the parameter 
space ($\alpha,\nu,\tau, r, a_e, a_i$) for some $k\neq$ and $\omega\neq0$.
We use \eqref{Turing1} to obtain the results presented in Figs.\:\ref{dispersion}, \ref{turing}, and \ref{turing2a}. 

In Fig.\:\ref{dispersion}, we show a dispersion relation (i.e., a $k$--$\omega$ curve) of the neural field for a particular set of values of the other parameters, i.e., $\alpha=5.0,\nu=1.0,\tau=0.75,r=5.0, a_e=10.0,$ and $ a_i=a_e/r=2.0$.
\begin{figure}[]
\centering
 \includegraphics[width=8.0cm,height=5.0cm]{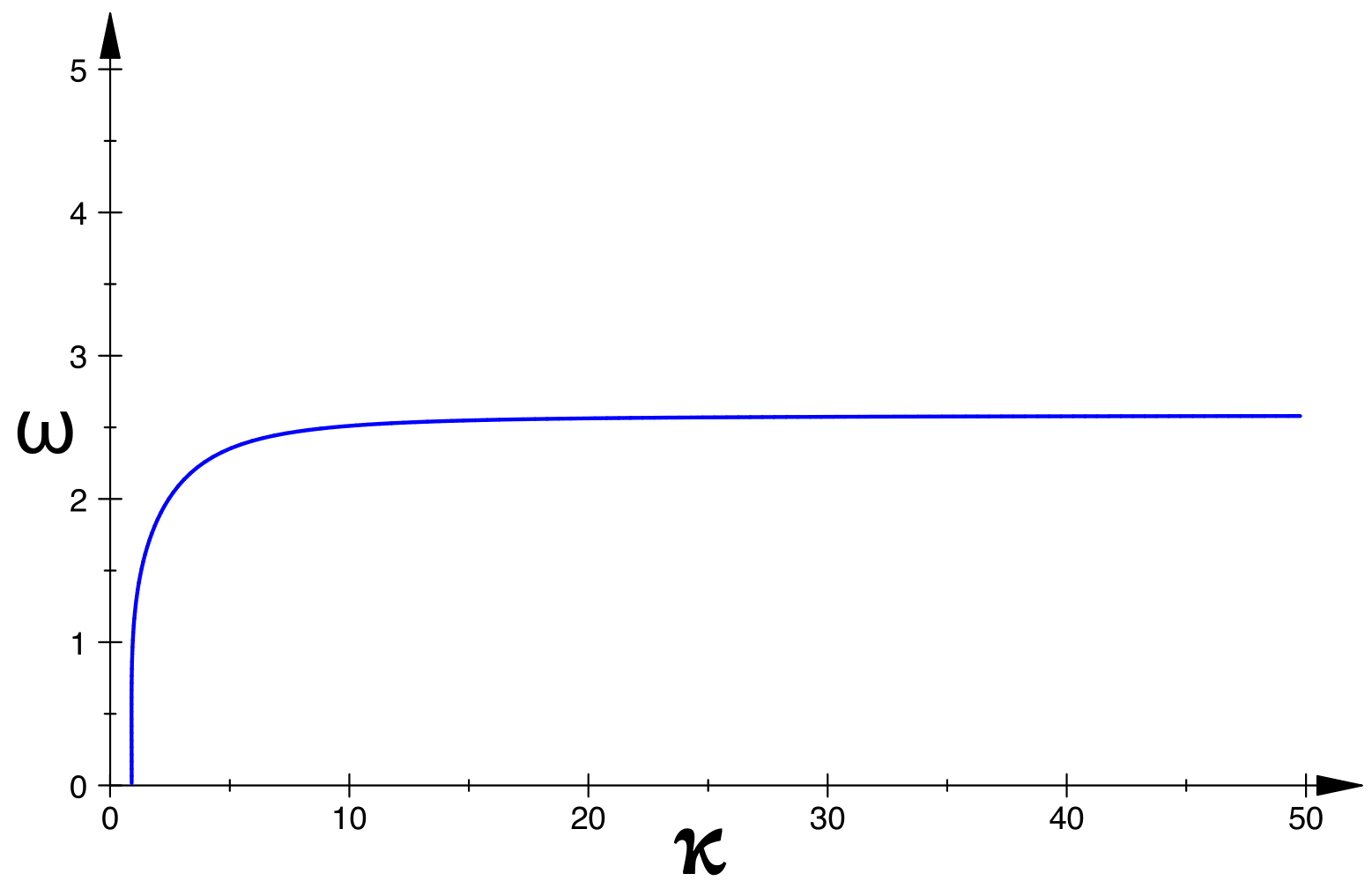}
\caption{A dispersion relation of the neural field satisfying the Turing-Hopf bifurcation equation given by \eqref{Turing1} for a fixed set of parameters values: $\alpha=5.0$, $\nu=1.0$, $\tau=0.75$, $r=5.0$, $a_{e}=10.0$, $a_i=a_e/r=2.0$.
}
\label{dispersion}
\end{figure}
In Fig.\:\ref{turing}\textbf{(a)}--\textbf{(d)}, we respectively show the Turing-Hopf bifurcation curves in different parameter spaces: 
($\alpha$-$\omega$), ($\nu$-$\omega$), ($\tau$-$\omega$), and ($r$-$\omega$) for a fixed value of the spatial mode $k=25.0$.
We note from Fig.\:\ref{turing}\textbf{(a)} and \textbf{(c)} that the memory decay $\alpha$ and the leakage parameter $\tau$ have opposite effects. The former increases the frequency $\omega$ of the oscillations, whereas the latter decreases it.  From Fig.\:\ref{turing}\textbf{(b)}, the transmission speed $\nu$ has a nonmonotonic influence, with a minimum for $\omega$ at a particular value of $\nu$. Similarly, Fig.\:\ref{turing}\textbf{(d)}, the ratio of the excitatory and inhibitory synaptic weights $r$ has a nonmonotonic influence on $\omega$. However, when the excitatory and inhibitory synaptic weights are balanced, i.e., when $r=a_e/a_i=1$, \eqref{Turing1} has many trivial solutions, that is, there exist infinitely many $\omega$ values that satisfy \eqref{Turing1}. This explains the vertical line in Fig.\:\ref{turing}\textbf{(d)} and Fig.\:\ref{turing2a}\textbf{(d)} at $r=1$.
\begin{figure}[]
\centering
\textbf{(a)}\includegraphics[width=7.0cm,height=4.5cm]{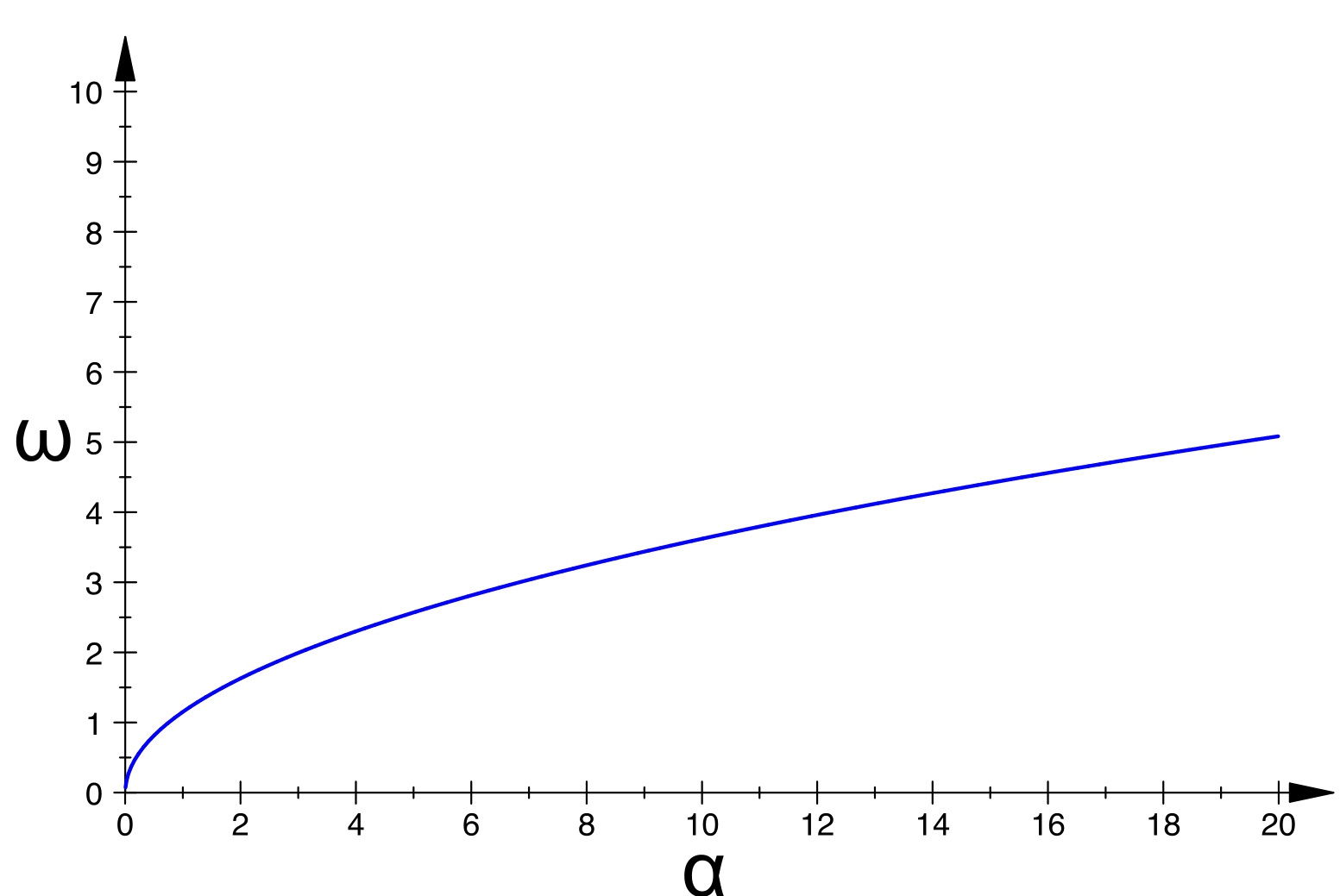}\textbf{(b)}\includegraphics[width=7.0cm,height=4.5cm]{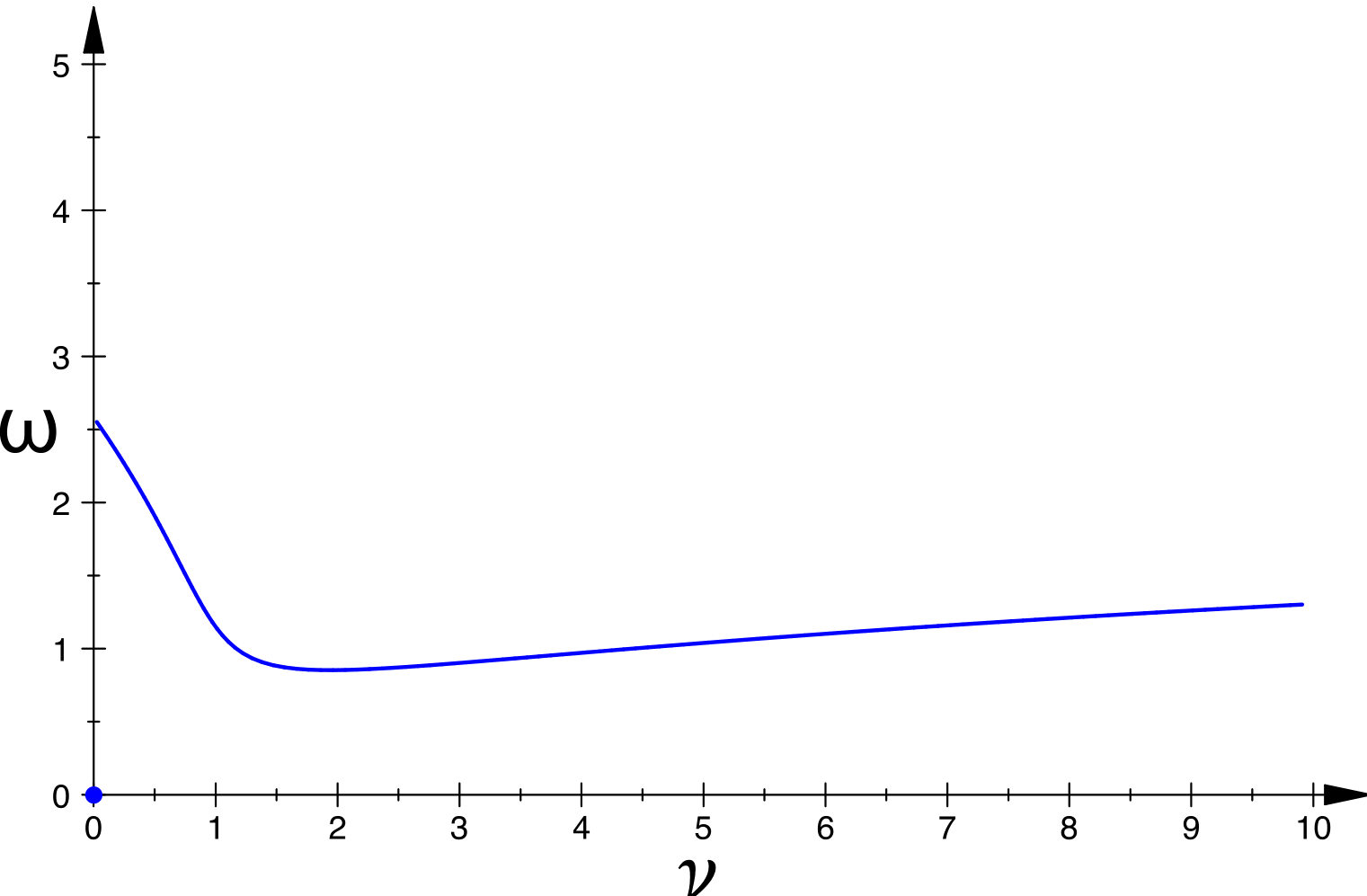}\\
\textbf{(c)}\includegraphics[width=7.0cm,height=4.5cm]{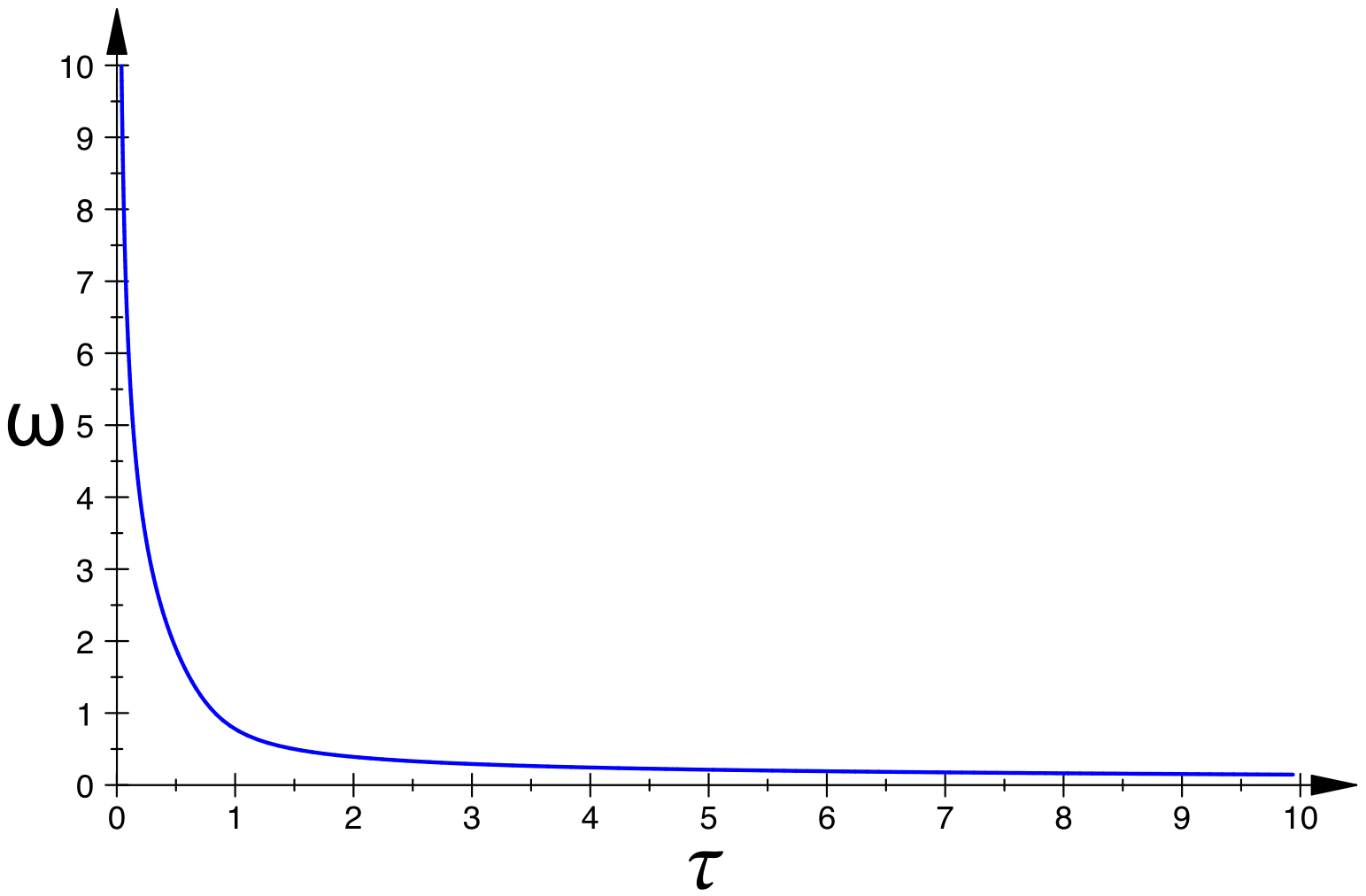}\textbf{(d)}\includegraphics[width=7.0cm,height=4.5cm]{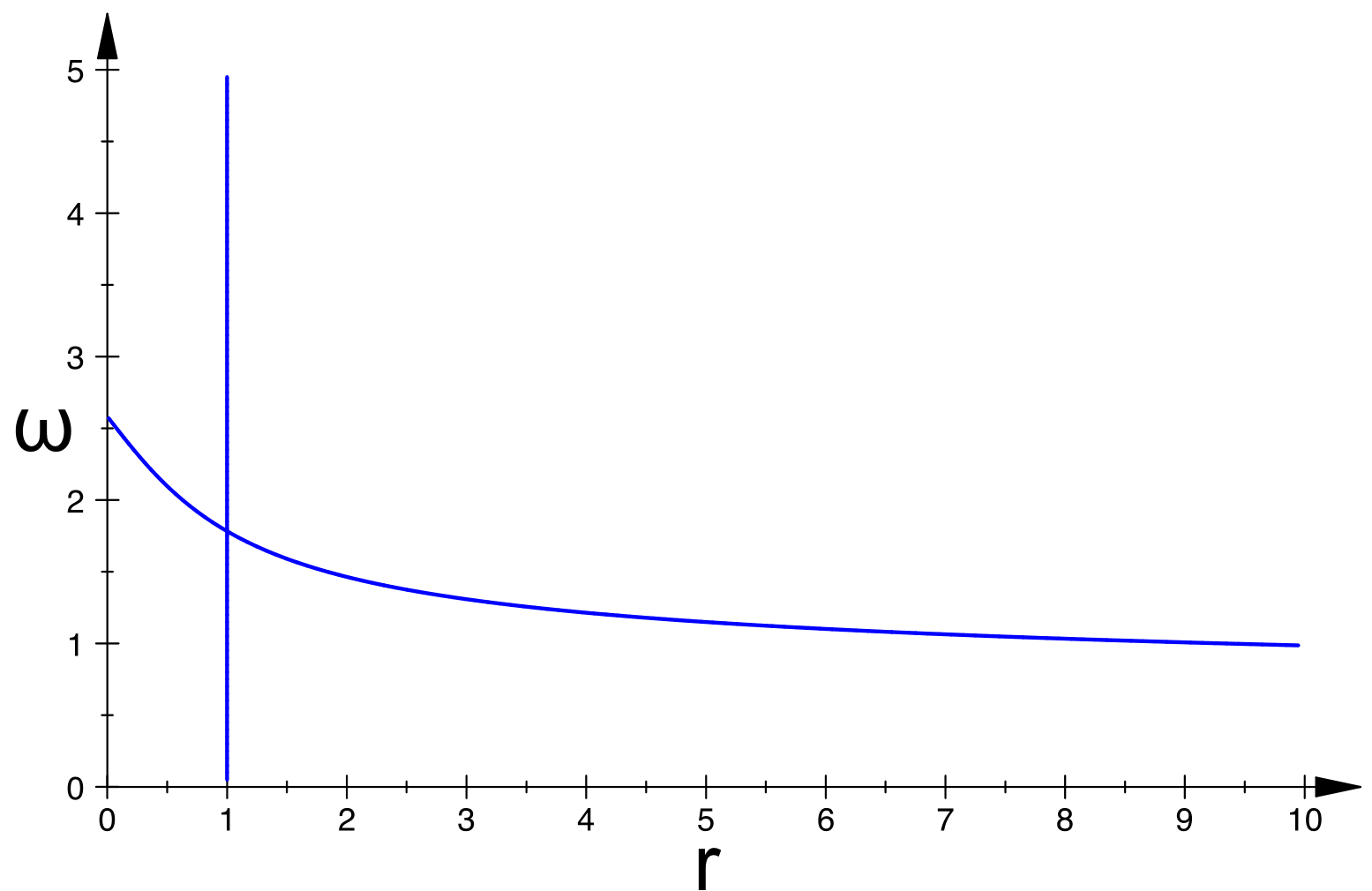}\\
\caption{The curves in panels \textbf{(a)}-\textbf{(d)}
represent the Turing-Hopf bifurcation curves in \eqref{Turing1} in
the ($\alpha$-$\omega$), ($\nu$-$\omega$), ($\tau$-$\omega$), and ($r$-$\omega$) planes, respectively. 
Parameter values are: \textbf{(a)} $k=25.0$, $\nu=1.0$, $\tau=0.75$, $r=5.0$; \textbf{(b)} $k=1.0$, $\alpha=5.0$, $\tau=0.75$,  $r=5.0$;
\textbf{(c)} $k=1.0$, $\alpha=5.0$, $\nu=1.0$, $r=5.0$; \textbf{(d)} $k=1.0$, $\alpha=5.0$, $\nu=1.0$, $\tau=0.75$. 
The remaining parameters are fixed at $a_{e}=10.0$, $a_i=a_e/r$.} \label{turing}
\end{figure}

In Fig.\:\ref{fig3}, we display the space-time patterns in three distinct regions 
of the Turing-Hopf bifurcation curve of Fig.\:\ref{turing}\textbf{(a)}, for example. Here, one can see, as expected, 
that the Turing-Hopf bifurcation leads to spatially and temporally non-constant solutions. 
With values fixed at $k=25.0$, $\nu=1.0$, $\tau=0.75$, $r=5.0$ $a_e=10.0$, and $a_i=a_e/r=2.0$, we chose a value for the exponential temporal kernel $\alpha$ and then calculate the corresponding temporal mode $\omega$ from \eqref{Turing1}, such that both values ($\alpha,\omega$) lie 
\textit{above} (as in Fig.\:\ref{fig3}\textbf{(a)}), \textit{on} (as in Fig.\:\ref{fig3}\textbf{(b)}), and \textit{below} (as in Fig.\:\ref{fig3}\textbf{(c)}) the Turing-Hopf bifurcation curve in Fig.\:\ref{turing}\textbf{(a)}. Comparing the patterns in Fig.\:\ref{fig3}\textbf{(a)}, \textbf{(b)}, and \textbf{(c)}, one can see that for a total time interval of 10 units, 
there is a change in the number of temporal oscillations, while the number of spatial oscillations does not change (because, of course, the spatial mode is fixed at $k=25.0$).
In Fig.\:\ref{fig3}\textbf{(a)}, with the values of $\omega$ and $\alpha$ lying \textit{above} the Turing-Hopf bifurcation curve, 
the neural field admits temporal oscillations with a relatively high frequency (12 oscillations per 10 units of time, i.e., $1.2$ hertz (Hz)).
In Fig.\:\ref{fig3}\textbf{(b)} with $\omega$ and $\alpha$ lying  \textit{on} the Turing-Hopf bifurcation curve, 
the frequency is reduced to 0.6 Hz, and in Fig.\:\ref{fig3}\textbf{(c)}, with 
$\omega$ and $\alpha$ lying \textit{below} the Turing-Hopf bifurcation curve, the frequency is further reduced to 0.2 Hz.
\begin{figure}[H]
\centering
 \textbf{(a)}\includegraphics[width=7.0cm,height=5.0cm]{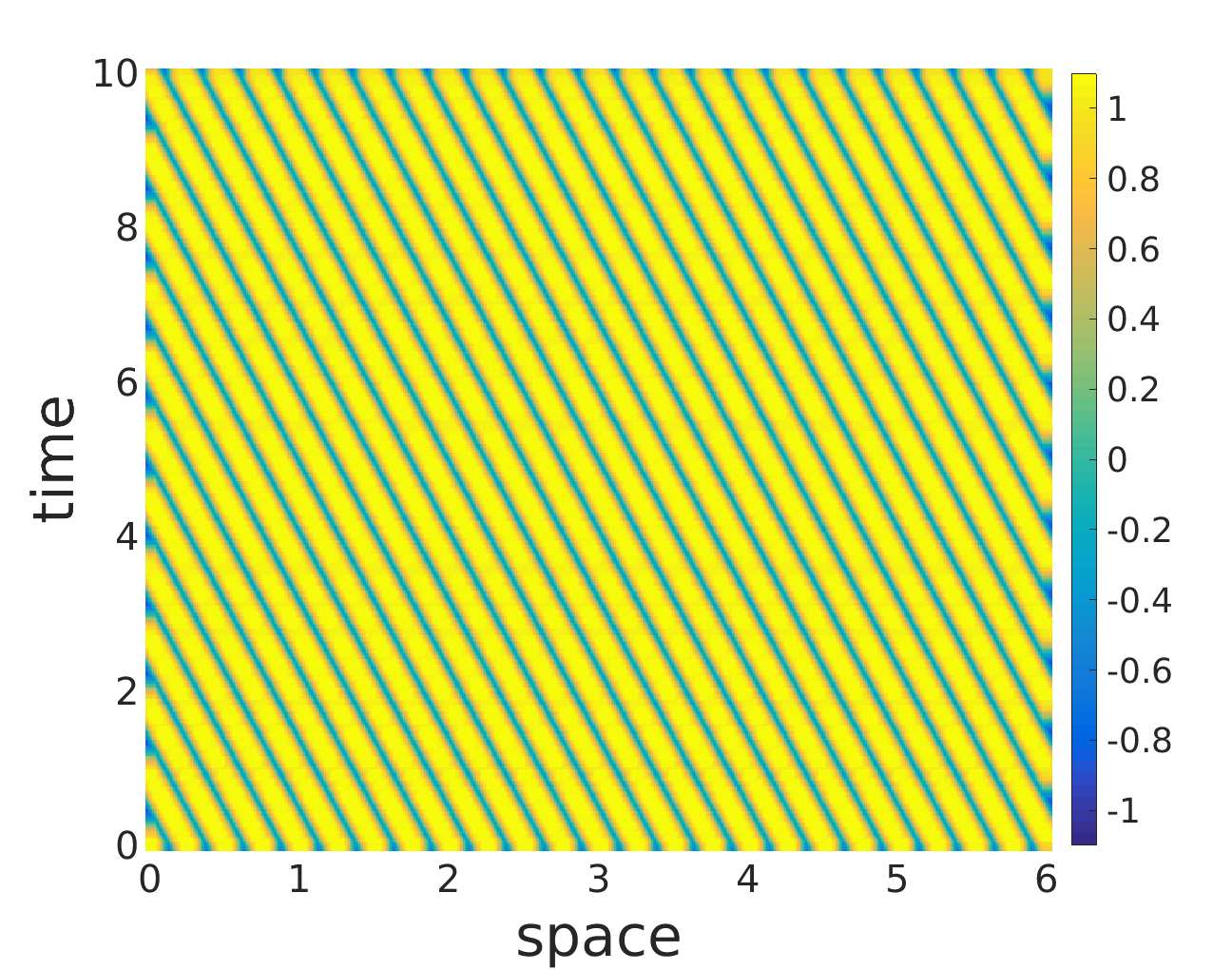}
 \textbf{(b)}\includegraphics[width=7.0cm,height=5.0cm]{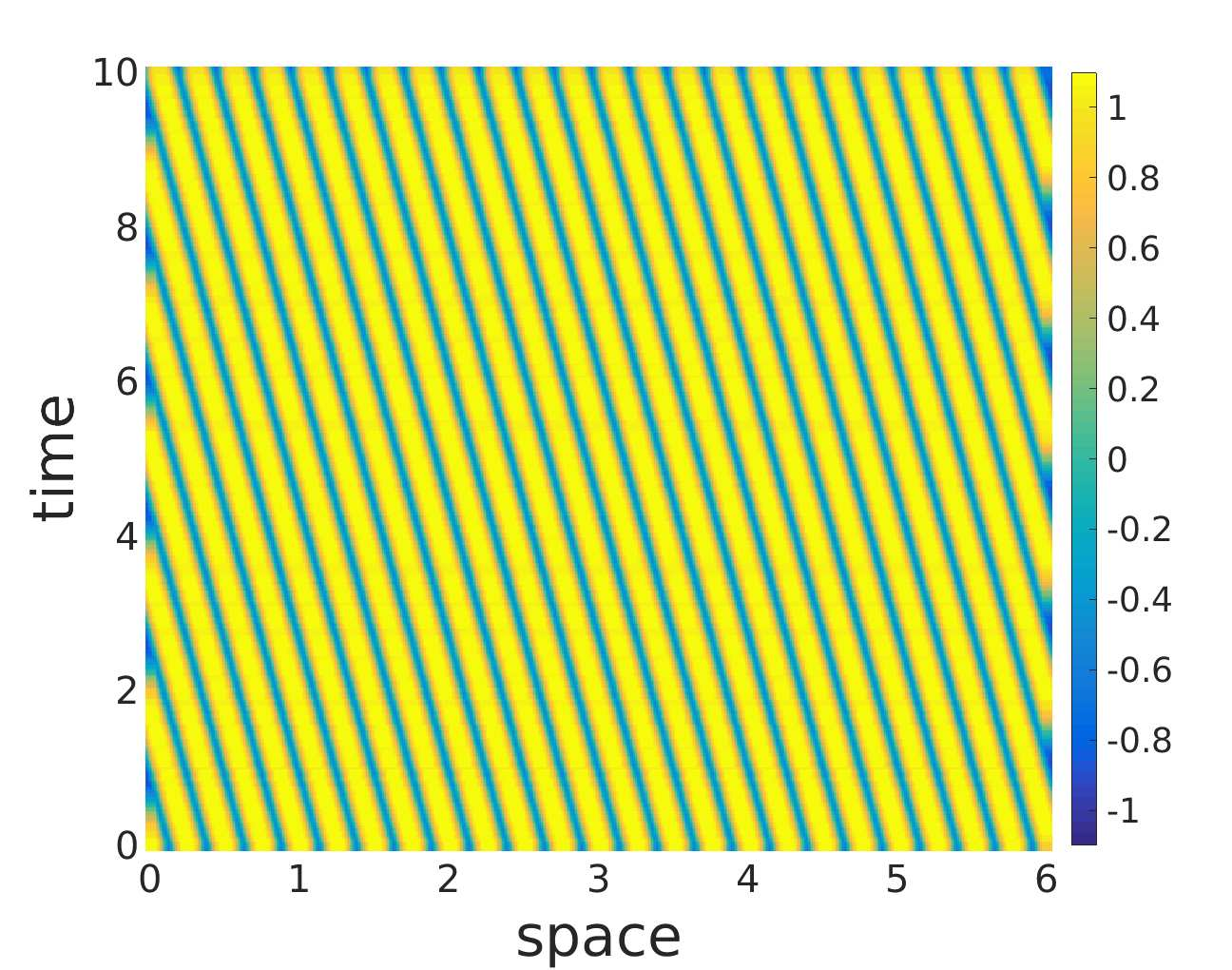}\\
 \textbf{(c)}\includegraphics[width=7.0cm,height=5.0cm]{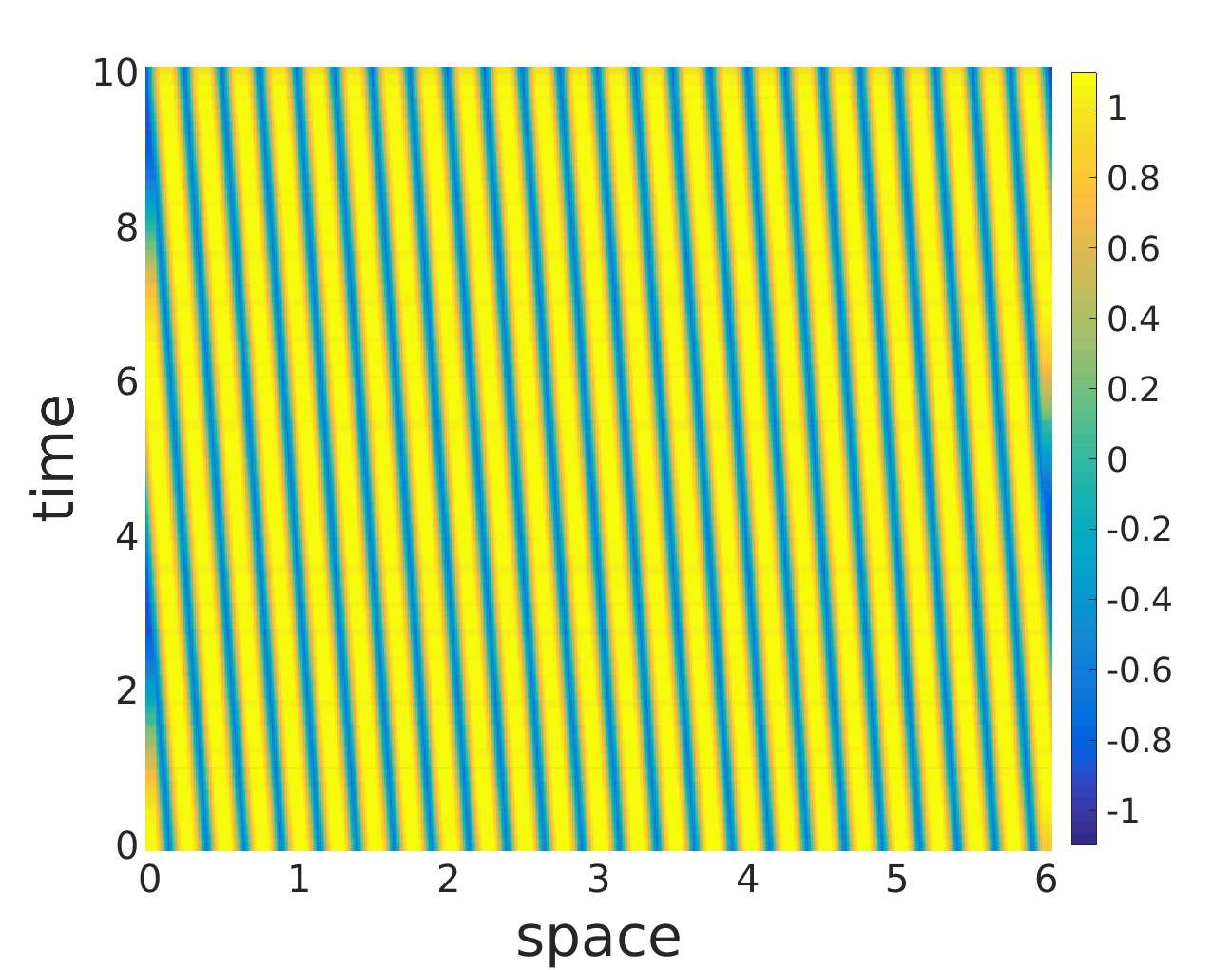}
\caption{Panels \textbf{(a)}-\textbf{(c)} show color-coded space-time patterns
  of the  membrane potential $v(x,t)$ emerging from a Turing-Hopf instability, 
leading to periodic oscillations of spatially and temporally non-constant solutions, 
obtained with the Gaussian connectivity kernel. Initial conditions are chosen
randomly from a uniform distribution on $[v_0-0.1, v_0+0.1]$. In panel \textbf{(a)}, $(\alpha,\omega)=(10.0,1.0)$ lies \textit{above} the 
Turing-Hopf bifurcation curve of Fig.\:\ref{turing}\textbf{(a)}. In \textbf{(b)}, $(\alpha,\omega)=(10.0,3.62)$ lies \textit{on} this 
Turing-Hopf bifurcation curve, and  in \textbf{(c)} $(\alpha,\omega)=(10.0,7.0)$ lies \textit{below} the curve. We observe a decrease in the frequency of the temporal oscillations of Turing-Hopf patterns from panel \textbf{(a)} to \textbf{(c)}, and a constant frequency in the 
spatial oscillations. Parameter values are: $a_e=10.0$, $r=5.0$, $a_i=a_e/r=2.0$, $\nu=1.0$, $\tau=0.75$, $k=25.0$.} \label{fig3}
\end{figure}

In Fig.\:\ref{turing2a}\textbf{(a)}-\textbf{(d)}, we show the Turing-Hopf bifurcation curves in different parameter spaces: 
($\alpha$-$k$), ($\nu$-$k$), ($\tau$-$k$), and ($r$-$k$), respectively, for a fixed value of the temporal mode $\omega=0.1$. We should contrast these relations with those of Fig.\:\ref{turing}. In effect, the dependence of the temporal and the spatial frequency values at the bifurcation on those other parameters is essentially the opposite.
\begin{figure}[]
\centering
 \textbf{(a)}\includegraphics[width=7.0cm,height=4.5cm]{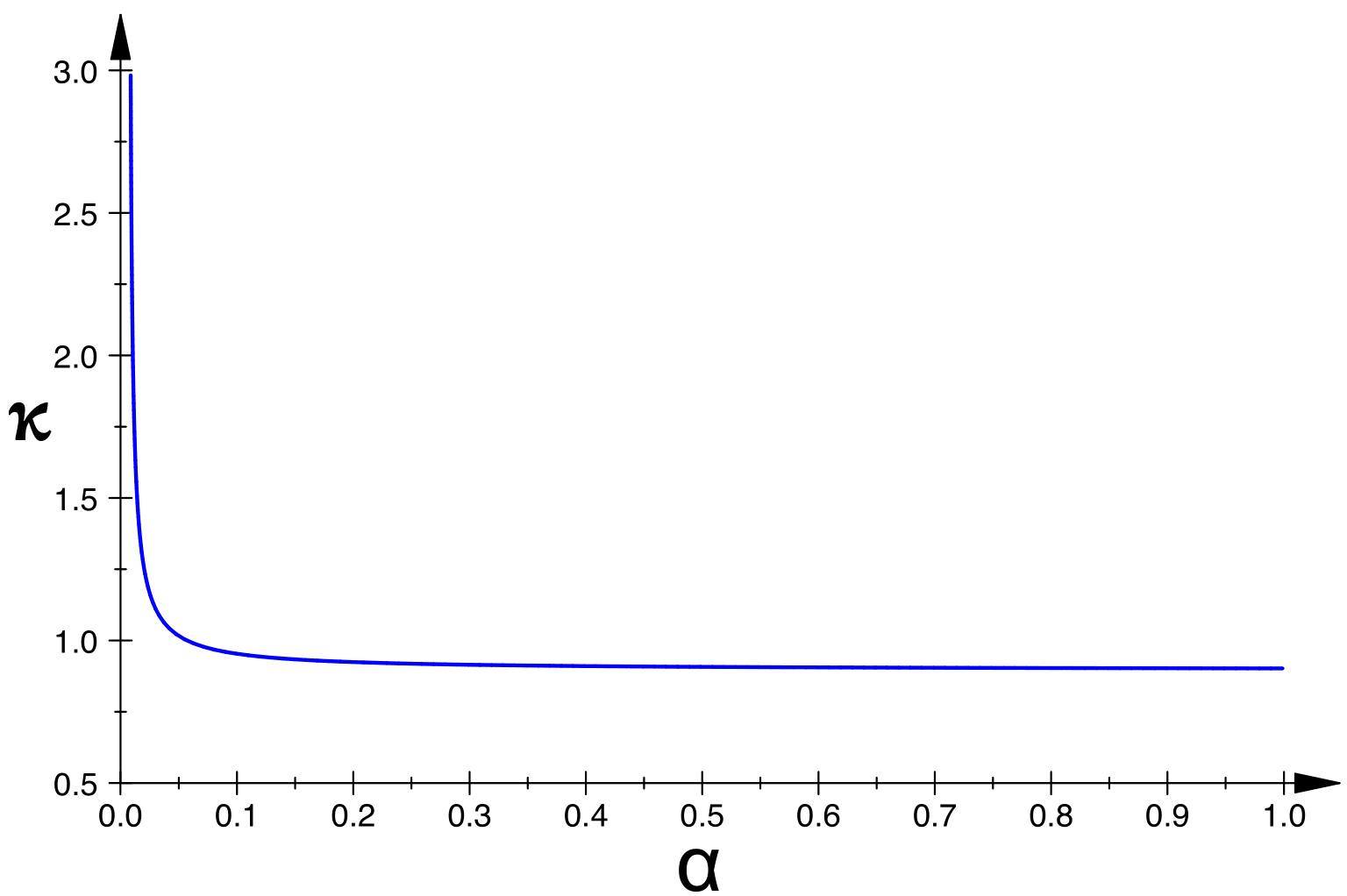}\textbf{(b)}\includegraphics[width=7.0cm,height=4.5cm]{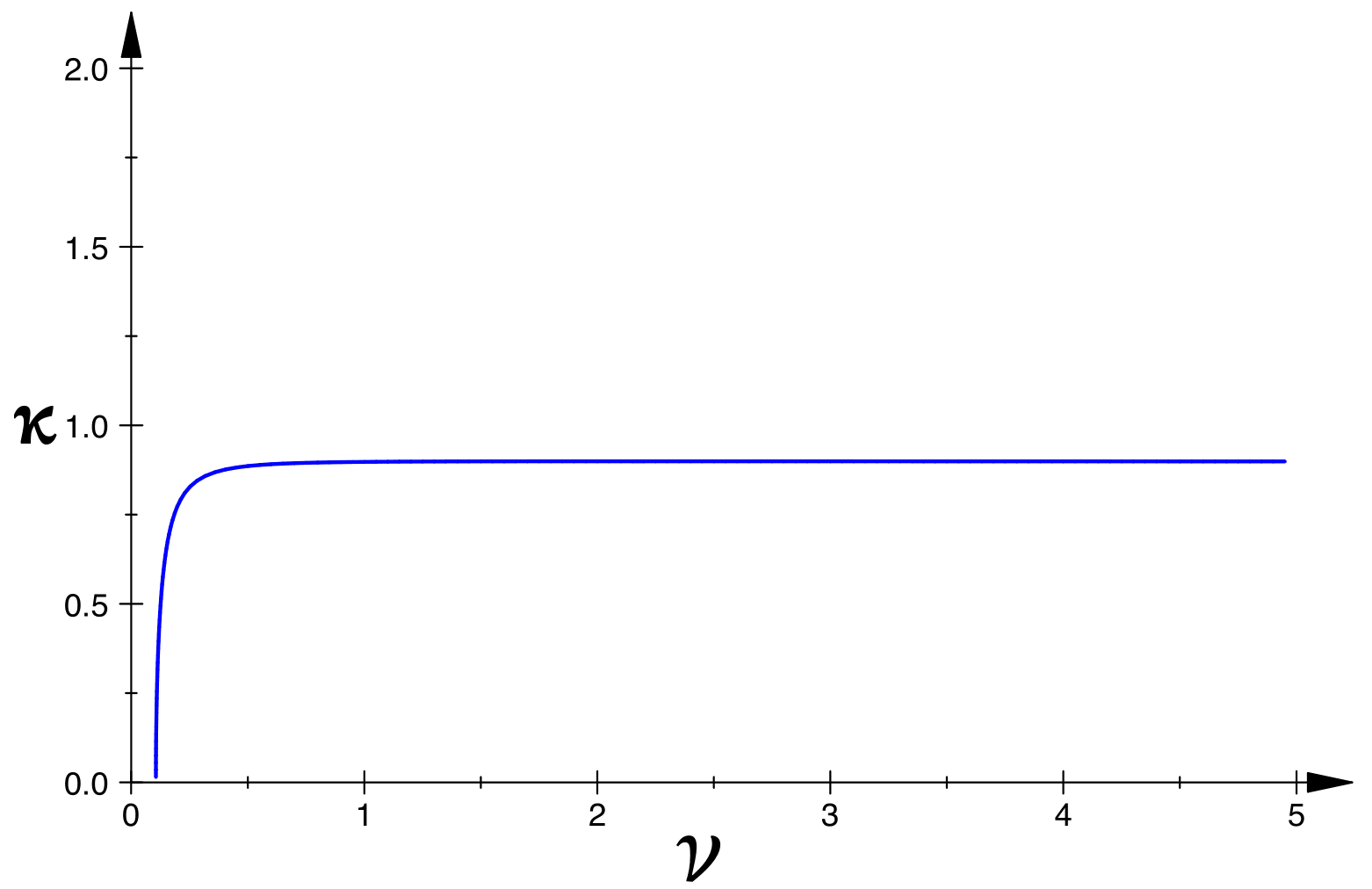}\\
 \textbf{(c)}\includegraphics[width=7.0cm,height=4.5cm]{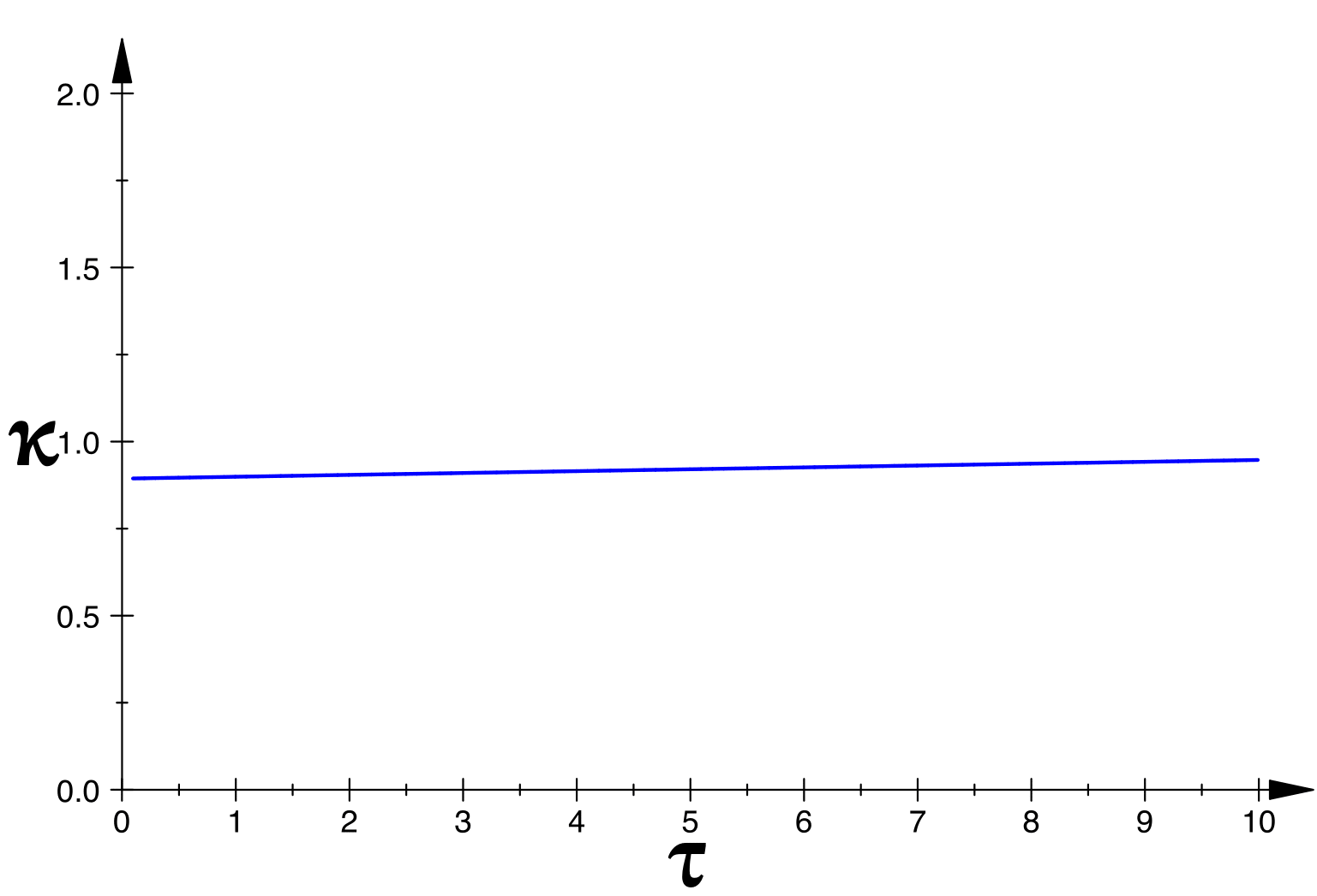}\textbf{(d)}\includegraphics[width=7.0cm,height=4.5cm]{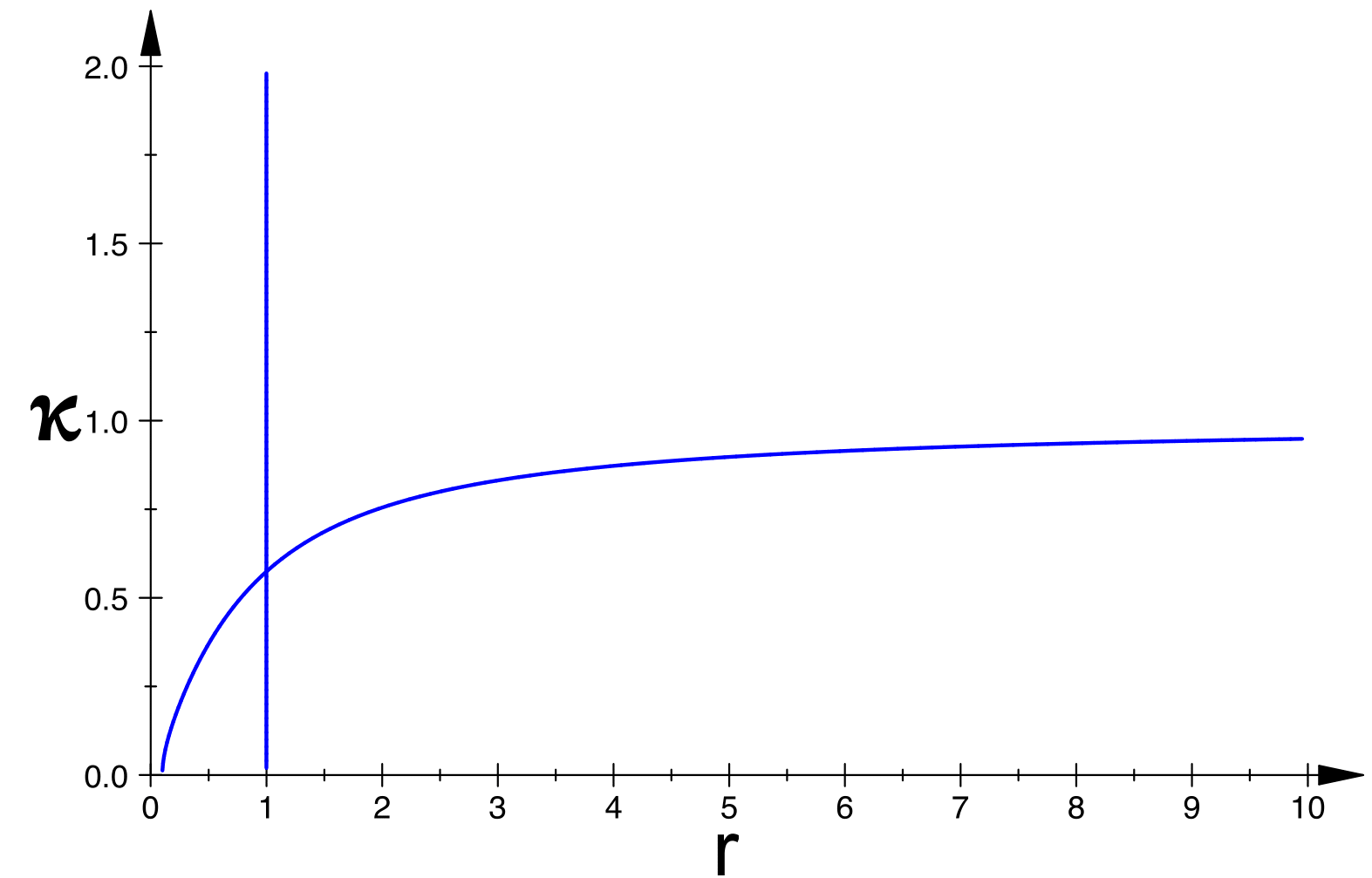}
\caption{The curves in panels \textbf{(a)}-\textbf{(d)}
represent the Turing-Hopf bifurcation curves in \eqref{Turing1} in
the ($\alpha$-$k$), ($\nu$-$k$), ($\tau$-$k$), and ($r$-$k$) planes, respectively. 
Parameter values are: \textbf{(a)} $\omega=0.1$, $\nu=1.0$, $\tau=0.75$, $r=5.0$; \textbf{(b)} $\omega=0.1$, $\alpha=5.0$, $\tau=0.75$,  $r=5.0$;
\textbf{(c)} $\omega=0.1$, $\alpha=5.0$, $\nu=1.0$, $r=5.0$; \textbf{(d)} $\omega=0.1$, $\alpha=5.0$, $\nu=1.0$, $\tau=0.75$. 
The remaining parameters are fixed at $a_{e}=10.0$, $a_i=a_e/r$.} \label{turing2a}
\end{figure}

In Fig.\:\ref{turing2b}, we display the space-time patterns in three distinct regions 
of the Turing-Hopf bifurcation curve of Fig.\:\ref{turing2a}\textbf{(a)}, for example. One can also see, 
that the Turing-Hopf bifurcation leads to spatially and temporally non-constant solutions. 
Here, for the sake of comparison, we have also fixed the parameters to the same values given in Fig.\:\ref{fig3}, i.e., $\nu=1.0$, $\tau=0.75$, $r=5.0$ $a_e=10.0$, $a_i=a_e/r=2.0$, 
and temporal mode parameter is fixed at $\omega=0.1$. As in Fig.\:\ref{fig3}, the patterns in Fig.\:\ref{turing2b} are obtained with values of the spatial frequency $k$ and the exponential temporal kernel $\alpha$, where we choose $\alpha$ and then calculate the corresponding spatial frequency $k$ from \eqref{Turing1}, such that both values 
($\alpha,k$) lie \textit{above} (as in Fig.\:\ref{turing2b}\textbf{(a)}), \textit{on} (as in Fig.\:\ref{turing2b}\textbf{(b)}), and \textit{below} (as in Fig.\:\ref{turing2b}\textbf{(c)}) the Turing-Hopf bifurcation curve in Fig.\:\ref{turing2a}\textbf{(a)}.

Comparing the panels in Fig.\:\ref{turing2b}, one can see a change in the patterns already observed in Fig.\:\ref{fig3},
but in terms of the frequency of the spatial oscillations, for a total space interval of 50 units and a temporal mode fixed at $\omega=0.1$. 
In Fig.\:\ref{turing2b}\textbf{(a)}, $k$ and $\alpha$  which are \textit{above} the Turing-Hopf bifurcation curve in Fig.\:\ref{turing2a}\textbf{(a)},
and the neural field oscillates with a relatively high spatial frequency,  i.e., 0.26 Hz.
In Fig.\:\ref{turing2b}\textbf{(b)}, $k$ and $\alpha$ lie \textit{on} the Turing-Hopf bifurcation curve, 
and the spatial frequency of oscillation is reduced by 0.16 Hz. In Fig.\:\ref{turing2b}\textbf{(c)},
$k$ and $\alpha$ lie \textit{below} the Turing-Hopf bifurcation curve and spatial frequency is further reduced to 0.1 Hz. 
However, in terms of the wavelengths in both space and time, the space-time patterns of the Turing-Hopf bifurcation 
in Fig.\:\ref{fig3} and Fig.\:\ref{turing2b} are different: Fig.\:\ref{turing2b} shows fewer oscillations on larger space and longer time intervals than in Fig.\:\ref{fig3}.

\begin{figure}[H]
\centering
 \textbf{(a)}\includegraphics[width=7.0cm,height=5.0cm]{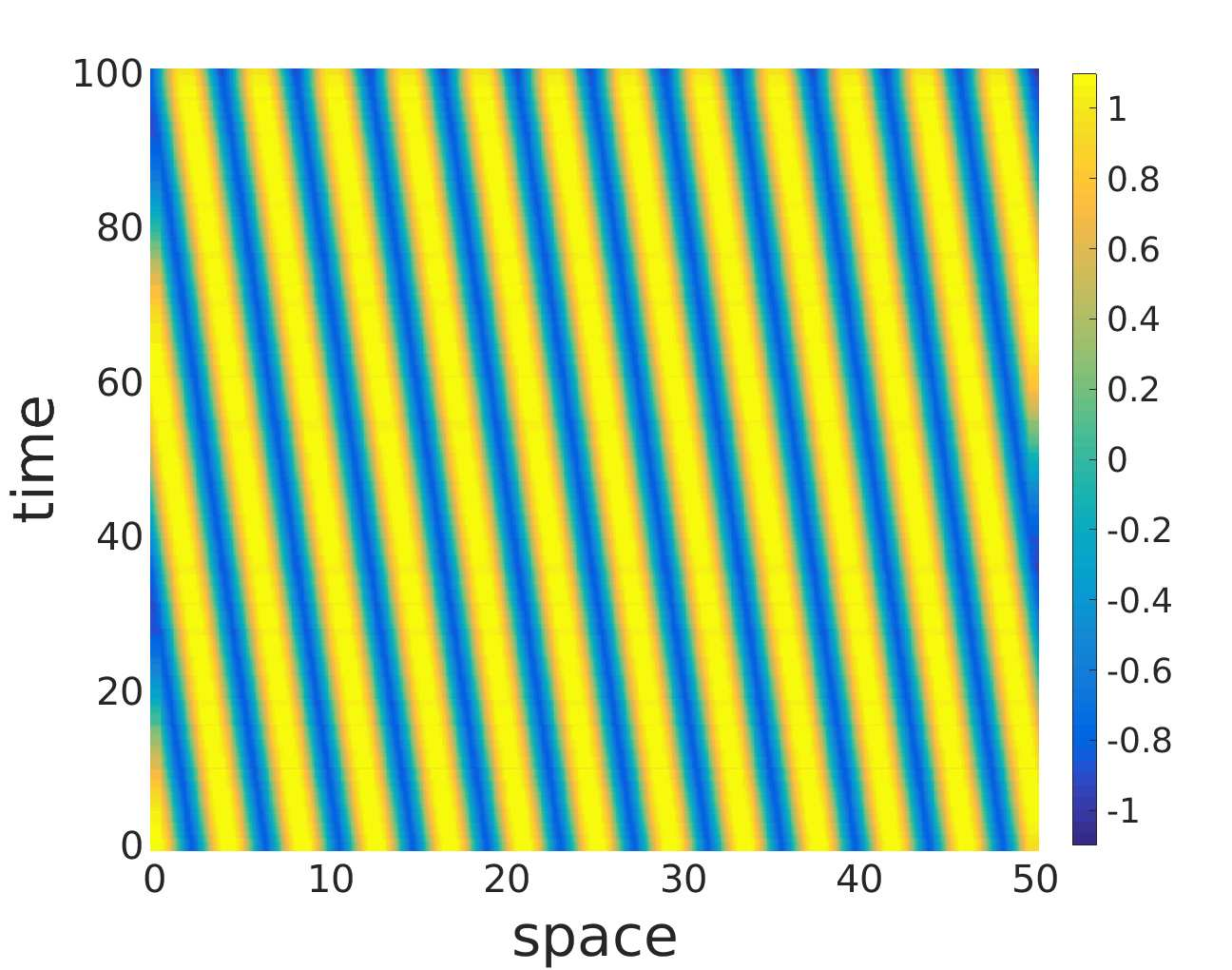}\textbf{(b)}\includegraphics[width=7.0cm,height=5.0cm]{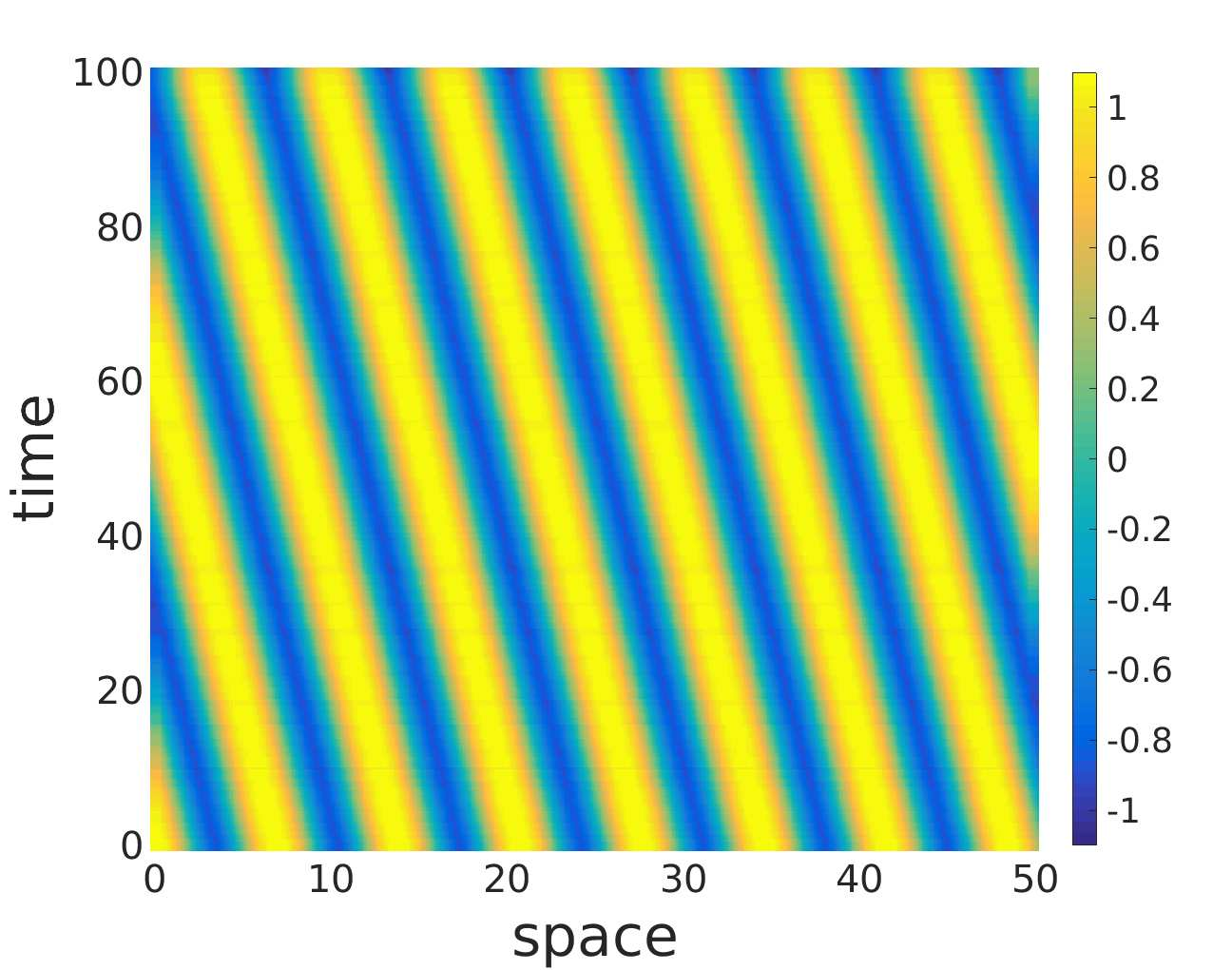}\\
 \textbf{(c)}\includegraphics[width=7.0cm,height=5.0cm]{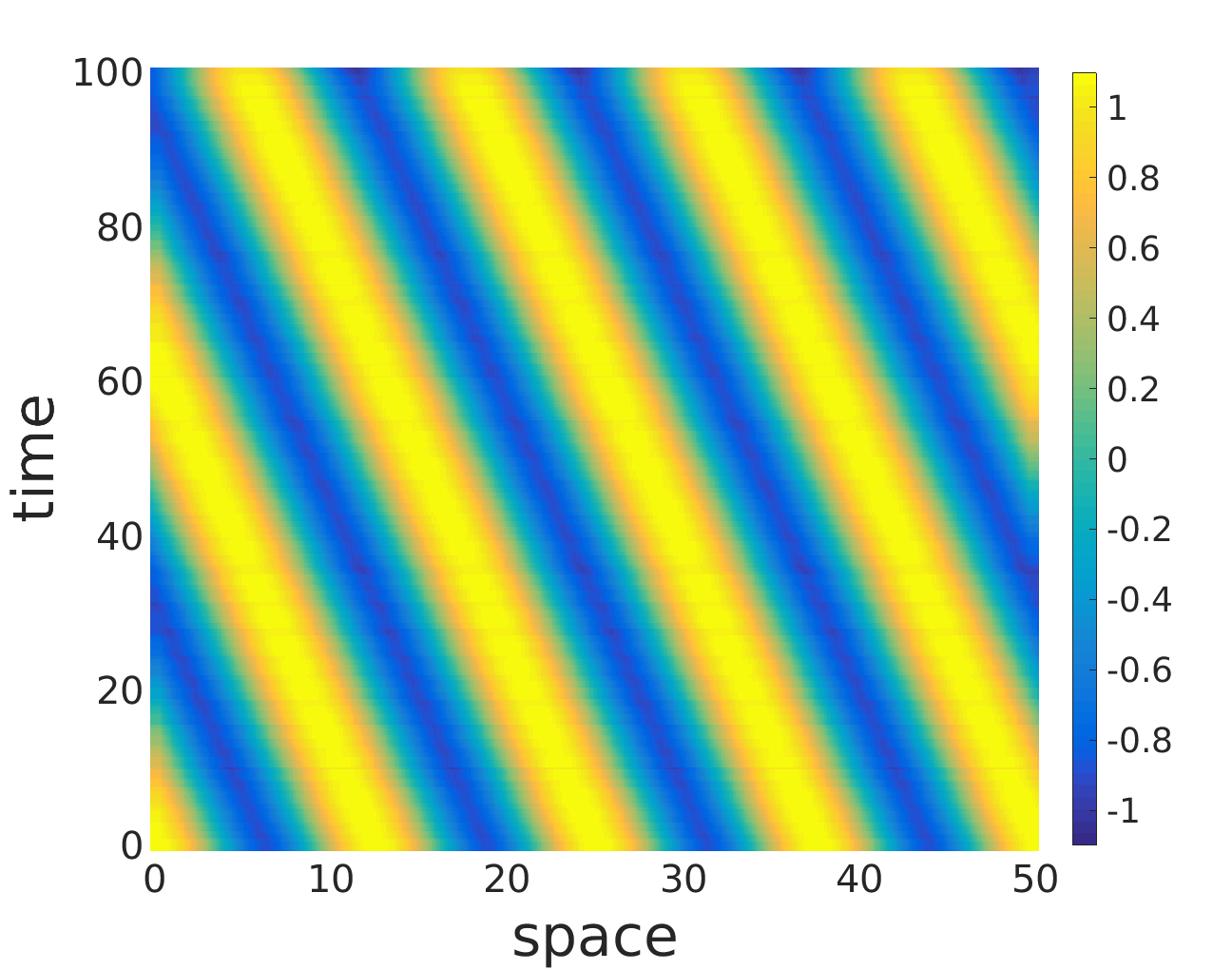}
\caption{Panels \textbf{(a)}-\textbf{(c)} show color coded space-time patterns of membrane potential $v(x,t)$ emerging from Turing-Hopf instability.
leading to periodic oscillations of spatially and temporally non-constant solutions obtained with the Gaussian connectivity kernel. Initial conditions are chosen
randomly from a uniform distribution on $[v_0-0.1, v_0+0.1]$. In panel \textbf{(a)}, $(\alpha,k)=(0.5,0.5)$ lies \textit{above} the 
Turing-Hopf bifurcation curve of Fig.\:\ref{turing2a}\textbf{(a)}. In \textbf{(b)}, $(\alpha,k)=(0.5,0.907)$ lies \textit{on} this 
Turing-Hopf bifurcation curve, and  $(\alpha,k)=(0.5,1.5)$ lies \textit{below} this curve. We observe a decrease in the 
frequency of the spatial oscillations of Turing-Hopf patterns from panel \textbf{(a)} to \textbf{(c)}, and a constant frequency in the 
temporal oscillations. Parameter values are: $a_e=10.0$, $r=5.0$, $a_i=a_e/r=2.0$, $\nu=1.0$, $\tau=0.75$, $\omega=0.1$.} \label{turing2b}
\end{figure}
\section{Summary and concluding remarks}\label{section5}
In this paper, we have studied the bifurcation behavior and the wave patterns generated by 
a neural field equation with an exponential temporal kernel. The exponential 
temporal kernel in \eqref{eq00}  takes into account 
the finite memory of past activities of the neurons, which the Green's function utilized in \citep{atay2004stability} does not. 
Our first observation was that  static bifurcations, such 
as saddle-node and pitchfork, as well as static Turing patterns, are not possible with an exponential temporal kernel, because the characteristic polynomial does not have an eigenvalue 0.
This is in contrast to  \citep{atay2004stability}, 
where the temporal kernel was taken as the Green's function rather than an exponential function, and thus allowed zero eigenvalues.  In analyzing the dynamic bifurcations of the equilibrium solution, we have obtained the conditions 
for the occurrence of Hopf and Turing-Hopf bifurcations. Furthermore, we have numerically illustrated these dynamic bifurcations with bifurcation diagrams and space-time patterns. 

Neural fields by now are an old paradigm in computational neuroscience. They were intended to generate spatiotemporal patterns at a level above individual neurons that may possibly underlie cognitive behavior, or more precisely, support, at least in qualitative terms, neurophysiological models of 
cognition, like those of \citep{malsburg1981,abeles1982} and many subsequent ones. In order to support cognition, such patterns should be qualitatively diverse and temporally flexible. That is, the model should allow for rapid switches between different cognitive states. In terms of a dynamical model, such switches should occur as bifurcations, depending on parameters that can be readily tuned. This has motivated our study. And, of course, the models should include neurophysiologically plausible mechanisms. In that regard, we have assumed transmission delays and exponentially decaying memory, both of which are neurophysiologically well-supported. And with these assumptions, we could indeed produce qualitatively diverse and temporally flexible patterns, most notably Turing-Hopf bifurcations. These generate patterns that are non-constant in time and space. Periodic oscillations or patterns that propagate in time are at the basis of the synchronization model first advocated in \citep{malsburg1981} and at the synfire chain model of \citep{abeles1982}. Such dynamical modes should be easily triggered, but also easily terminated, and therefore require fine-tuning of a bifurcation parameter in the models.  On the other hand, there also needs to be spatial patterns so as to enable cognitive processing to make distinctions, as required, for instance, for feature binding, and as supported, for instance, by \citep{gray1989stimulus} and many subsequent studies.  Importantly, this should not depend on anatomical differences between brain regions, but such patterns should occur within specific regions. A simplifying, but in this context reasonable, assumption is a homogeneous field of neurons, as in neural field models, and the question then is to understand how such a field can generate behavior that is both temporally and spatially inhomogeneous, which model assumptions support this, and how this arises through bifurcations.  

It is worth pointing out that it is not necessarily the case that the results presented in our work can directly (i.e., without a further and detailed investigation that may even require completely different mathematical tools than those used here) be taken into account to interpret the bifurcation dynamics of any given modified neural mean field model. For example, a stochastic neural field equation \citep{touboul2012mean,faugeras2015stochastic,bressloff2019stochastic} with exponential temporal kernel and leakage term will almost surely behave differently and require stochastic analysis \citep{faugeras2015stochastic}, which is beyond the methods used in our deterministic neural mean field equation. Thus, a future and more comprehensive research study is necessary to analyze the stochastic static and dynamics bifurcations of a neural mean field equation with exponential temporal kernel and leakage term.  

\begin{acknowledgements}
This work was supported by the Max-Planck-Institut f\"{u}r Mathematik in den Naturwissenschaften, Leipzig, Germany.
\end{acknowledgements}

%

\begin{thebibliography}{35}
\providecommand{\natexlab}[1]{#1}
\providecommand{\url}[1]{\texttt{#1}}
\expandafter\ifx\csname urlstyle\endcsname\relax
  \providecommand{\doi}[1]{doi: #1}\else
  \providecommand{\doi}{doi: \begingroup \urlstyle{rm}\Url}\fi


\bibitem{wu2008}  J.Y. Wu, X.Y. Huang, C. Zhang, Propagating waves of activity in the neocortex: what they are, what they do, Neuroscientist 14 (2008) 487-502.

\bibitem{townsend2015} R.G. Townsend, S.S. Solomon, S.C. Chen, A.N. Pietersen, P.R. Martin, S.G. Solomon, P. Gong, Emergence of complex wave patterns in primate cerebral cortex. Journal of Neuroscience 35 (2015) 4657-4662.

\bibitem{lubenov2009} E.V. Lubenov, A.G. Siapas,  Hippocampal theta oscillations are traveling waves, Nature 459 (2009) 534-539.

 \bibitem{malsburg1981} C. von der Malsburg, The correlation theory of brain function Models of Neural Networks vol II, ed E Domany, J van Hemmen and K Schulten, New York, Springer, (1994).
 
 
\bibitem{abeles1982} M. Abeles, Studies of brain function: Vol. 6. Local cortical circuits:
An electro-physiological study. Berlin, Springer, (1982).

\bibitem{galinsky2020}  V.L. Galinsky, L.R. Frank, Brain Waves: Emergence of Localized, Persistent, Weakly Evanescent Cortical Loops, J. Cogn. Neurosci. 32 (2020) 2178-2202.



\bibitem {kandel2000principles} E.R. Kandel, J.H. Schwartz, T.M. Jessell, Department of~Biochemistry,
  Molecular Biophysics~Thomas Jessell, S. Siegelbaum, A.J. Hudspeth, Principles of neural science, vol~4, McGraw-hill, New York, (2000).

\bibitem{wilson1973mathematical}
H.R. Wilson, J.D. Cowan, A mathematical theory of the functional dynamics of cortical and
  thalamic nervous tissue, Biol. Cybernet. 13 (1973) 55-80.
  
\bibitem{wilson1972excitatory}
H.R. Wilson, J.D. Cowan, Excitatory and inhibitory interactions in localized populations of
  model neurons, Biophys. J. 12 (1972) 1-24.

\bibitem{amari1977dynamics}
S.-I. Amari, Dynamics of pattern formation in lateral-inhibition type neural
  fields, Biol. Cybernet. 27 (1977) 77-87.

\bibitem{veltz2011stability}
R. Veltz, O. Faugeras, Stability of the stationary solutions of neural field equations with
  propagation delays, J. Math. Neurosci. 1 (2011) 1.

\bibitem{perlovsky2006toward}
L.I. Perlovsky, Toward physics of the mind: Concepts, emotions, consciousness, and
  symbols, Phys. Life Rev. 3 (2006) 23--55.

\bibitem{alswaihli2018kernel}
J. Alswaihli, R. Potthast, I. Bojak, D. Saddy, A. Hutt, Kernel reconstruction for delayed neural field equations, J. Math. Neurosci. 8 (2018) 3.

\bibitem{abbassian2012neural}
A.H. Abbassian, M. Fotouhi, M. Heidari, Neural fields with fast learning dynamic kernel, Biol. Cybernet. 106 (2012) 15--26.

\bibitem{bressloff2011spatiotemporal}
P.C. Bressloff, Spatiotemporal dynamics of continuum neural fields, J. Phys. A. Math. Theor. 45 (2011) 033001.

\bibitem{haken2007brain}
H. Haken, Brain dynamics: an introduction to models and simulations,
    Springer-Verlag, Berlin, (2007).

\bibitem{karbowski2000multispikes}
J. Karbowski, N. Kopell, Multispikes and synchronization in a large neural network with
  temporal delays, Neural Comput. 12 (2000) 1573-1606.

\bibitem{morelli2004associative}
L.G. Morelli, G. Abramson, M.N. Kuperman, Associative memory on a small-world neural network, Eur. Phys. J. B 38 (2004) 495-500.

\bibitem{prager2003stochastic}
T. Prager, L.S Geier, Stochastic resonance in a non-markovian discrete state model for
  excitable systems, Phys. Rev. Lett. 91 (2003) 230601.

\bibitem{spiridon2001effect}
M. Spiridon, W. Gerstner, Effect of lateral connections on the accuracy of the population code
  for a network of spiking neurons, Network 12 (2001) 409-421.

\bibitem{gerstner2002spiking}
W.Gerstner, W.Kistler, Spiking neuron models, Cambridge Univ.Press, (2002).

\bibitem{atay2004stability}
F.M. Atay, A. Hutt, Stability and bifurcations in neural fields with finite propagation
  speed and general connectivity, SIAM J. Appl. Math. 65 (2004) 644-666.

\bibitem{atay2006neural}
F.M. Atay, A. Hutt, Neural fields with distributed transmission speeds and long-range
  feedback delays, SIAM J. Appl. Dyn. Syst. 5 (2006) 670-698.

\bibitem{hutt2005analysis}
A. Hutt, F.M. Atay, Analysis of nonlocal neural fields for both general and
  gamma-distributed connectivities, Physica D 203 (2005) 30-54.

\bibitem{hutt2006effects}
A. Hutt, F.M. Atay, Effects of distributed transmission speeds on propagating activity in
  neural populations, Phys. Rev. E 73 (2006) 021906.
\bibitem{veltz2011stability}
R. Veltz, and O. Faugeras, Stability of the stationary solutions of neural field equations with propagation delays. The Journal of Mathematical Neuroscience 1 (2011)  1-28.

\bibitem{spek2022dynamics}
L. Spek,  K. Dijkstra,  S.A. van Gils, and M. Polner, Dynamics of delayed neural field models in two-dimensional spatial domains. Journal of differential equations 317 (2022) 439-473.


\bibitem{van2013local}S.A. van Gils, S.G.Janssens, Y.A.Kuznetsov,  and S.Visser, On local bifurcations in neural field models with transmission delays. Journal of mathematical biology 66 (2013)  837-887.


\bibitem{muller2018cortical}
L. Muller, F. Chavane, J. Reynolds, T.J. Sejnowski, Cortical travelling waves: mechanisms and computational principles. Nature Reviews Neuroscience, 19 (2018) 255-268.

\bibitem{watt2009traveling}
A.J. Watt, H. Cuntz, M. Mori, Z. Nusser, P.J. Sj\"ostr\"om, M. H\"ausser, Traveling waves in developing cerebellar cortex mediated by asymmetrical Purkinje cell connectivity. Nature Neuroscience 12 (2009) 463-473.

\bibitem{senk2020conditions} J.Senk, K. Korvasov{\'a}, J. Schuecker, E. Hagen, T. Tetzlaff, M.
Diesmann, M. Helias, Conditions for wave trains in spiking neural networks, 
Physical review research 2 (2020) 023174.
  
\bibitem{polner2017space}
M. Polner, J.J.W. Van der Vegt, S.V. Gils, A space-time finite element method for neural field 
equations with transmission delays, SIAM J. Sci. Comput. 39 (2017) B797-B818.

\bibitem{arqub2017adaptation}
O. A. Arqub, Adaptation of reproducing kernel algorithm for solving fuzzy
  fredholm--volterra integrodifferential equations, Neural Comput. Appl. 28 (2017) 1591-1610.

\bibitem{faugeras2015stochastic}
O. Faugeras, J. Inglis, Stochastic neural field equations: a rigorous footing, J. Math. Biol. 71 (2015) 259-300.

\bibitem{rankin2014continuation}
J. Rankin, D. Avitabile, J. Baladron, G. Faye, D.J. Lloyd, Continuation of localized coherent structures in nonlocal neural
  field equations, SIAM J. Sci. Comput. 36 (2014) B70-B93.

\bibitem{fang2016monotone}
J. Fang, G. Faye, Monotone traveling waves for delayed neural field equations, Math. Models Methods Appl. Sci.
  26 (2016) 1919-1954.

\bibitem{breakspear2017dynamic}
M. Breakspear, Dynamic models of large-scale brain activity, Nat. Neurosci. 20 (2017) 340.

\bibitem{pinto2001spatially}
D.J. Pinto and G.B. Ermentrout, Spatially structured activity in synaptically coupled neuronal
  networks: I. traveling fronts and pulses, SIAM J. Appl. Math. 62 (2001) 206-225

\bibitem{nunez1995experimental}
P.L. Nunez, Neocortical dynamics and human EEG rhythms, (pp. 534-590), New York, Oxford University Press (1995).

\bibitem{coombes2007modeling}
S. Coombes, N. Venkov, L. Shiau, I. Bojak, D.T. Liley, C.R.
  Laing, Modeling electrocortical activity through improved local
  approximations of integral neural field equations, Phys. Rev. E 76 (2007) 051901.

\bibitem{hutt2003pattern}
A. Hutt, M. Bestehorn, T. Wennekers, Pattern formation in intracortical neuronal fields, Network
 14 (2003) 351-368.
 
 \bibitem{robinson1997propagation} P. A. Robinson, C. J. Rennie, and J. J. Wright, Propagation and stability of waves of
electrical activity in the cerebral cortex, Phys. Rev. E 56 (1997) 826-840.

\bibitem{folias2005breathers}
S.E. Folias, P.C. Bressloff, Breathers in two-dimensional neural media, Phys. Rev. Lett. 95 (2005) 208107.

\bibitem{laing2005spiral}
C.R. Laing, Spiral waves in nonlocal equations, SIAM J. Appl. Dyn. Syst. 4 (2005) 588-606.

\bibitem{folias2004breathing}
S.E. Folias, P.C. Bressloff, Breathing pulses in an excitatory neural network, SIAM J. Appl. Dyn. Syst. 3
  (2004) 378-407.

\bibitem{venkov2007dynamic}
N.A. Venkov, S. Coombes, P.C. Matthews, Dynamic instabilities in scalar neural field equations with
  space-dependent delays. Physica D 232 (2007) 1-15.

\bibitem{touboul2012mean}
J. Touboul, Mean-field equations for stochastic firing-rate neural fields with
  delays: Derivation and noise-induced transitions, Physica D 241 (2012) 1223-1244.

\bibitem{bojak2010axonal}
I. Bojak, D.T. Liley, Axonal velocity distributions in neural field equations. PLoS Comput. Biol. 6 (2010) e1000653.

\bibitem{hutt2008additive}
A. Hutt, A. Longtin, L. Schimansky-Geier, Additive noise-induced turing transitions in spatial systems with
  application to neural fields and the Swift-Hohenberg equation. Physica D 237 (2008) 755-773.

\bibitem{coombes2005waves}
S. Coombes, Waves, bumps, and patterns in neural field theories. Biol. Cybernet. 93 (2005) 91-108.


\bibitem{gray1989stimulus}
C.M. Gray, W. Singer, Stimulus-specific neuronal oscillations in orientation columns of cat visual cortex.
Proc. Nati. Acad. Sci. 86 (1989) 1698-1702. 

\bibitem{faugeras2015stochastic}
O. Faugeras, J. Inglis, Stochastic neural field equations: a rigorous footing. J. Math. Biol. 71 (2015) 259–300. 

\bibitem{bressloff2019stochastic}
P.C. Bressloff, Stochastic neural field model of stimulus-dependent variability in cortical neurons. PLoS Computational Biology 15 (2019) e1006755.





\end{thebibliography}
\end{document}